\documentclass[a4paper, article, oneside, USenglish, final]{memoir}

\usepackage[utf8]{inputenx} 
\usepackage[T1]{fontenc}    

\usepackage{lmodern}           
\usepackage[scaled]{beramono}  
\usepackage[final]{microtype}  
\pretitle{\begin{center}\LARGE\sffamily\bfseries}                       
\setsecheadstyle{\large\bfseries\boldmath\sffamily\raggedright}         
\setsubsecheadstyle{\normalsize\bfseries\boldmath\sffamily\raggedright} 
\setsubsubsecheadstyle{\normalsize\bfseries\sffamily\raggedright}
\setparaheadstyle{\normalsize\bfseries\sffamily\raggedright}            
\setsecnumdepth{subsection}

\usepackage{amssymb}   
\usepackage{amsthm}    
\usepackage{thmtools}  
\usepackage{mathtools} 
\usepackage{mathrsfs}  
\numberwithin{equation}{chapter}

\RequirePackage{enumitem}
\setlist[itemize]{font = \upshape, before = \leavevmode}
\setlist[enumerate]{font = \upshape, before = \leavevmode}
\setlist[description]{font = \bfseries\sffamily, before = \leavevmode}

\usepackage{array}
\usepackage{longtable}

\usepackage{xspace}    
\usepackage{graphicx}  
\graphicspath{{figures/}}
\usepackage{babel}     
\usepackage{csquotes}  
\usepackage{textcomp}  
\usepackage{listings}  
\usepackage[
            textwidth = 40mm,
            textsize = footnotesize]{todonotes}
\usepackage[notref, notcite]{showkeys}

\RequirePackage[backend     = biber,
                sortcites   = true,
                giveninits  = true,
                maxbibnames = 5,
                doi         = false,
                isbn        = false,
                url         = false,
                style       = alphabetic]{biblatex}
\DeclareNameAlias{sortname}{family-given}
\DeclareNameAlias{default}{family-given}
\DeclareFieldFormat[article]{volume}{\bibstring{jourvol}\addnbspace#1}
\DeclareFieldFormat[article]{number}{\bibstring{number}\addnbspace#1}
\renewbibmacro*{volume+number+eid}
{
    \printfield{volume}
    \setunit{\addcomma\space}
    \printfield{number}
    \setunit{\addcomma\space}
    \printfield{eid}
}

\usepackage{varioref}
\usepackage[colorlinks, allcolors = uiolink, linktoc = all, pdfusetitle]{hyperref}
\usepackage[nameinlink, capitalize, noabbrev]{cleveref}

\DeclareSymbolFont{extraup}{U}{zavm}{m}{n}
\DeclareMathSymbol{\vardiamond}{\mathalpha}{extraup}{87}

\declaretheoremstyle[headfont   = \bfseries\sffamily,
                     notefont   = \normalfont,
                     bodyfont   = \itshape,
                     spaceabove = 6pt,
                     spacebelow = 6pt]{plain}
\declaretheoremstyle[headfont   = \bfseries\sffamily,
                     notefont   = \normalfont,
                     spaceabove = 6pt,
                     spacebelow = 6pt]{definition}
\declaretheorem[style = plain, numberwithin = chapter]{theorem}
\declaretheorem[style = plain,      sibling = theorem]{corollary}
\declaretheorem[style = plain,      sibling = theorem]{lemma}
\declaretheorem[style = plain,      sibling = theorem]{proposition}
\declaretheorem[style = plain,      sibling = theorem]{conjecture}
\declaretheorem[style = definition, sibling = theorem]{definition}

\declaretheorem[style = definition, sibling = theorem, qed = \(\spadesuit\)]{remark}
\declaretheorem[style = definition, numbered = no]{acknowledgements}

\DeclareMathOperator{\Sing}{Sing}

\DeclarePairedDelimiter{\linspan}{\langle}{\rangle}

\newcommand{\R}{\mathbb{R}}   
\newcommand{\C}{\mathbb{C}}   
\renewcommand{\P}{\mathbb{P}} 
\newcommand{\PP}{\mathbb{P}}  

\definecolor{uiolink}{HTML}{0B5A9D}

\newcommand{\locus}[1]{rank-\(#1\)-locus}
\newcommand{\point}[1]{rank-\(#1\)-point}
\newcommand{\quadric}[1]{rank-\(#1\)-quadric}
\newcommand{\rankmatrix}[1]{rank-\(#1\)-matrix}

\let\oldand\and
\renewcommand\and{\texorpdfstring{\oldand}{,}}

\title{Rational quartic spectrahedra}
\author{Martin Hels{\o} \and Kristian Ranestad}
\date{\vskip-3em}

\begin{document}

\maketitle

\begin{abstract}
    \noindent
    Rational quartic spectrahedra in $3$-space are semialgebraic convex subsets in $\R^3$ of semidefinite, real symmetric $(4 \times 4)$-matrices,
    whose boundary admits a rational parameterization.
    The Zariski closure in $\C\P^3$ of the boundary of a rational spectrahedron
    is a rational complex symmetroid.
    We give necessary conditions on the configurations of singularities
    of the corresponding real symmetroids in $\R\P^3$
    of rational quartic spectrahedra.
    We provide an almost exhaustive list of examples realizing the configurations,
    and conjecture that the missing example does not occur.
\end{abstract}

\chapter{Introduction}

Spectrahedra are important basic objects in polynomial optimization and in convex algebraic geometry \cite{BPT}.
They are intersections of the cone of positive-semidefinite matrices in the space of real symmetric $(n \times n)$-matrices by an affine subspace.
Quartic spectrahedra are the case of $(4 \times 4)$-matrices intersected with a $3$-dimensional affine space that contains a positive definite matrix.
We identify the affine space with $\R^3$.
The boundary of a quartic spectrahedron has a Zariski closure $V(f_A) \subset \R\PP^3$ defined by the determinant $f_A(x) \coloneqq f_A(x_0, x_1, x_2, x_3)$ of a symmetric matrix $A(x)$,
where explicitly
\begin{equation}
    \label{eq:AxAx}
    A(x) \coloneqq A_0 x_0 + A_1 x_1 + A_2 x_2 + A_3 x_3, 
\end{equation}
and each $A_i$ is a real symmetric $(4 \times 4)$-matrix.
Since the matrix $A(x)$ is symmetric,
the surface $V(f_A)$ is called a \emph{(real) symmetroid}.
Similarly,
the complex algebraic boundary $V_\C(f_A) \subset \C\PP^3$
defined by $f_A$ is called a \emph{complex symmetroid}
to distinguish it from its real points $V(f_A)$.
We say that a real symmetroid $V(f_A)$ is \emph{spectrahedral}
if it is the Zariski closure of the boundary of a nonempty spectrahedron,
i.e.,
if $A(x_p)$ is definite for some $x_p \in \R^3$.
If the complex algebraic boundary of a quartic spectrahedron is a rational complex symmetroid,
we say that the spectrahedron is \emph{rational}.
The general quartic symmetroid is not rational.

For a general matrix $A(x)$,
the singular points of $V(f_A)$ is a finite set of double points,
quadratic singularities called \emph{nodes}.
The possible arrangements of the nodes of general quartic spectrahedral symmetroids
were identified by Degtyarev and Itenberg \cite{DI}
and further investigated by Ottem et al.\ in \cite{ORSV}.
In \cite{ORSV},
the authors describe their paper as a
\enquote{first step towards the classification
of all spectrahedra of a given degree and dimension}.
A quartic surface with a finite set of nodes,
or more generally rational double points,
is irrational,
in fact birational to a K3-surface.
This paper sets out to fill in part of the classification
of quartic spectrahedra in $3$-space by studying the rational members.

The significant feature of a rational spectrahedron
is that its boundary allows a parameterization of rational functions from a subset $X \subset \C\P^2$.
This is immediate from the fact that
the Zariski closure in $\C\P^3$ of the boundary
allows a parameterization of rational functions from $\C\P^2$.
We do not provide any explicit boundary parameterizations.
\emph{A priori},
$X$ does not have to be a subset of $\R\P^2$
even though the boundary of the spectrahedron is real.

A quartic surface is rational
only if it has a triple point, an elliptic double point or is singular along a curve cf.\ \cite{Jes,Noe}.
The first author identified families of rational quartic symmetroids in \cite{He}.
We state the results of that paper,
after making a note about ranks and quadrics.

At every point $x_p \in V(f_A) \subset \C\PP^3$,
the matrix $A(x_p)$ has rank at most $3$.
We say that $x_p \in V(f_A)$ is a \emph{\point{$k$}},
if $A(x_p)$ has rank $k$.
The symmetroid $V(f_A)$ has a double point at each \point{2},
and a triple point at each \point{1}.
It may, however, also be singular at \point{3}s.
This phenomenon is characterized in \cref{lem:accidental-node}
by properties of the web of quadrics associated to the symmetroid:

If $y \coloneqq (y_0, y_1, y_2, y_3)$,
then $q_A(x_p) \coloneqq y\cdot A(x_p) \cdot y^T$ is a quadratic form,
and its vanishing $Q_A(x_p) \coloneqq V(q_A(x_p)) \subset \C\PP^3_y$ is a quadric surface.
The set $Q_A(x) \coloneqq \big\{Q_A(x_p) \mid x_p \in \C\PP^3 \big\}$ is called a \emph{web of quadrics}.

\begin{lemma}[{\cites[Lemma~2.13]{Ili+17}[Lemma~1.1]{Wal81}}]
    \label{lem:accidental-node}
    The symmetroid $V_\C(f_A)$ has a singularity at a \point{3}~$x_p$
    if and only if the web of quadrics $Q_A(x)$ has a basepoint
    at the singular point of $Q_A(x_p)$.
\end{lemma}

\noindent
The main results of \cite{He} can be summarized as follows:

\begin{theorem}[{\cite{He}}]
    \label{thm:symmsing}
    The rational complex quartic symmetroids in $\C\P^3$ form irreducible families
    whose general members are of the following types:
    \begin{enumerate}[label = \Alph{*}.,
                        ref = \textup{\Alph{*}}]
        \item
        \label{item:triple}
        $V_\C(f_A)$ has a triple point and six additional nodes;

        \item
        \label{item:tacnode}
        $V_\C(f_A)$ has a tacnode and six additional nodes,
        and the web $Q_A(x)$ has two basepoints;

        \item
        \label{item:conic}
        $V_\C(f_A)$ is singular along a conic and has four additional nodes,
        and the web $Q_A(x)$ has four linearly independent basepoints;

        \item
        \label{item:line}
        $V_\C(f_A)$ has rank $2$ along a line and has six additional nodes,
        and the web $Q_A(x)$ has four coplanar basepoints;

        \item
        \label{item:rank3-line}
        $V_\C(f_A)$ is singular of rank $3$ along a line and has four additional nodes,
        and the web $Q_A(x)$ has one basepoint.
    \end{enumerate}
\end{theorem}

\begin{remark}
    For the general members in the families
    mentioned in \cref{thm:symmsing},
    i.e.,
    symmetroids of types~\ref{item:triple},
    \ref{item:tacnode}, \ref{item:conic}, \ref{item:line}, \ref{item:rank3-line},
    all the additional nodes are \point{2}s and they are isolated. 
    Non-general members can have more nodes that are \point{3}s,
    or some of the singularities may coincide.
    In this paper an isolated node will always be a \point{2}.
\end{remark}

\begin{remark}
    No symmetroids of type \ref{item:rank3-line} of \cref{thm:symmsing}
    are spectrahedral.
    A consequence of \cref{lem:accidental-node} is that
    singular \point{3}s appear in complex conjugate pairs
    on spectrahedral symmetroids.
    A spectrahedral symmetroid which is singular of rank~$3$ along a line,
    is therefore singular along two lines.
\end{remark}

\begin{proposition}
    \label{prop:general-accidental-lines}
    The complex quartic symmetroids in $\C\P^3$
    singular of rank~$3$ along two intersecting lines, $L_1$, $L_2$,
    form an irreducible family
    whose general members $V_\C(f_A)$ have a \point{2} at $L_1 \cap L_2$,
    two isolated nodes and the base locus of the web $Q_A(x)$
    is a scheme of length~$4$ with support in two points.
\end{proposition}

\begin{proof}
    The statement about the singularities is \cite[Proposition~9.3]{He}
    and the claim about the base locus
    follows from \cref{lem:accidental-node} and \cite[Lemma~2.6]{Hel19}.
\end{proof}

\begin{definition}
    \label{def:general-accidental-lines}
    A complex quartic symmetroid in $\C\P^3$
    singular of rank $3$ along two lines
    intersecting in a \point{2},
    with two additional, isolated nodes
    is said to be of \emph{type \textup{F}}.
\end{definition}

\noindent
In this paper,
we show with reference to \cref{thm:symmsing,def:general-accidental-lines}:

\begin{theorem}
    Let $S \coloneqq V(f_A) \subset \R\P^3$ be a spectrahedral symmetroid,
    and suppose that the complex symmetroid $V_\C(f_A) \subset \C\P^3$ is rational.
    Let $a \ge 0$ denote the number of real, isolated nodes on $S$
    and let $0 \le b \le a$ denote the number of nodes on the boundary of the spectrahedron.
    \label{thm:main}
    \begin{enumerate}[label = \Alph{*}.,
                        ref = \textup{\Alph{*}}]
        \item
        \label{main-item:triple}
        If $V_\C(f_A)$ is of type \textup{\ref{item:triple}},
        then $S$ has a triple point on the boundary of the spectrahedron
        and $0 \le b \le a \le 6$, with $a$ even.

        \item
        \label{main-item:tacnode}
        If $V_\C(f_A)$ is of type \textup{\ref{item:tacnode}},
        then either
        \begin{enumerate}[label = \arabic{*}),
                            ref = \textup{\Alph{enumi}.\arabic{*})}]
            \item
            \label{main-item:tacnode-on-boundary}
            $S$ has a tacnode
            on the boundary of the spectrahedron;

            \item
            \label{main-item:tacnode-disjoint}
            $S$ has a tacnode
            disjoint from the spectrahedron.
        \end{enumerate}
        In either case,
        $0 \le b \le a \le b + 2 \le 6$, both even.

        \item
        \label{main-item:conic}
        If $V_\C(f_A)$ is of type \textup{\ref{item:conic}},
        then either
        \begin{enumerate}[label = \arabic{*}),
                            ref = \textup{\Alph{enumi}.\arabic{*})}]
            \item
            \label{main-item:conic-on-boundary}
            $S$ is singular along a smooth conic section
            with a real point that lies
            on the boundary of the spectrahedron
            and $0 \le b \le a \le b + 2 \le 4$,
            both even and $a \ge 2$;

            \item
            \label{main-item:conic-disjoint}
            $S$ is singular along a smooth conic section
            with a real point that is
            disjoint from the spectrahedron
            and $a = b = 2$ or $a = b = 4$;

            \item
            \label{main-item:cyclide}
            $S$ is singular along a smooth conic section
            with no real points
            and $a = b = 2$.
        \end{enumerate}

        \item
        \label{main-item:line}
        If $V_\C(f_A)$ is of type \textup{\ref{item:line}},
        then $S$ has rank~$2$ along a line
        disjoint from the boundary of the spectrahedron
        and $0 \le b \le a \le b + 2 \le 6$, both even.

        \addtocounter{enumi}{1}
        \item
        \label{main-item:two-lines}
        If $V_\C(f_A)$ is of type \textup{F},
        then either
        \begin{enumerate}[label = \arabic{*}),
                            ref = \textup{F.\arabic{*})}]
            \item
            \label{main-item:two-lines-on-boundary}
            $S$ is singular of rank~$3$ along two conjugated intersecting lines
            whose intersection point lies on the boundary of the spectrahedron;

            \item
            \label{main-item:two-lines-disjoint}
            $S$ is singular of rank~$3$ along two conjugated intersecting lines
            whose intersection point is disjoint from the spectrahedron.            
        \end{enumerate}
        In either case, $0 \le a = b \le 2$ even.
    \end{enumerate}
\end{theorem}

\noindent
\cref{thm:main} provides necessary conditions for the pair $(a, b)$
to be realized by real, rational spectrahedral symmetroids.
This is proven case-by-case in \cref{sec:real-singularities,sec:proof}.
After that, in \cref{sec:degenerations},
we discuss deformation relations between the symmetroids appearing in \cref{thm:main}.
We show that symmetroids of type~\ref{item:tacnode}
degenerate into a symmetroid of type~\ref{main-item:two-lines},
but that no other types degenerate into each other.

Sufficient conditions for the pair $(a, b)$
are given by explicit examples of rational spectrahedral symmetroids.
\cref{app:examples} contains examples
realizing all solutions of the bounds given by \cref{thm:main},
except spectrahedral symmetroids of type~\ref{main-item:line}
with $(a, b) = (0, 0)$.
We conjecture that they do not exist:

\begin{conjecture}
    A quartic symmetroid in $\C\P^3$
    with a nonempty spectrahedron and a line of \point{2}s has---%
    or is a degeneration of one with---at least two real, isolated nodes.
\end{conjecture}

\acknowledgements{We would like to thank Jan Stevens for the proof of \cref{prop:conic-real-nodes}, and the anonymous referee for helpful comments that improved the exposition of the paper.}

\chapter{Real singularities of rational spectrahedral symmetroids}
\label[section]{sec:real-singularities}

Before we can prove \cref{thm:main},
we need some preliminary results.
We restrict the attention to real rational quartic symmetroids $V(f_A)$ with a nonempty spetrahedron,
i.e., with $A(x_p)$ definite for some $x_p\in \R^3$.
First note that if $A(x_p)$ is definite,
then $Q_A(x_p)$ has no real points.
So if the quartic spectrahedron of $A(x)$ is nonempty,
then the web of quadratic surfaces $Q_A(x)$ has no common real points,
i.e., no real basepoints, so they have an even number of complex conjugate basepoints.
Therefore,
when we consider real singularities of the symmetroid $V(f_A)$,
they represent real \quadric{2}s in a web of quadrics $Q_A(x)$ with complex basepoints.
We say that a real quadric is \emph{semidefinite} resp.\ \emph{indefinite},
when the associated symmetric matrix is.
We begin by specializing the base loci mentioned in \cref{thm:symmsing}
to pairs of complex conjugated points.

\begin{lemma}
    \label{lem:line-definite}
    Let $p_1$, $\overline{p}_1$, $p_2$, $\overline{p}_2$,
    be two pairs of complex conjugate points in $\C\PP^3$ that do not all lie in a line.
    Then a real \quadric{2} that contains both pairs of points is indefinite
    if and only if it contains the real lines $\linspan{p_1,\overline{p}_1}$ and $\linspan{p_2,\overline{p}_2}$.
\end{lemma}

\begin{proof}
    Assume that $Q \coloneqq M\cup N$ is a real \quadric{2},
    the union of two planes $M$ and $N$.
    If $M$ and $N$ are both real,
    then $Q$ is indefinite,
    while if $M$ and $N$ are complex conjugates,
    then $Q$ is semidefinite.
    
    If $M$ and $N$ each contains only two of the four points
    $p_1$, $\overline{p}_1$, $p_2$, $\overline{p}_2$, the lemma follows.
    If $M$ contains exactly three of the points,
    say  $p_1$, $\overline{p}_1$, $p_2$ and is not real,
    then $N$ must contain $p_1$, $\overline{p}_1$, $\overline{p}_2$,
    so $Q$ is semidefinite.
    If $M$ contains all four basepoints,
    $M$ is real, so $Q$ is indefinite.
\end{proof}

\begin{lemma}
    \label{lem:quadric-of-quadrics}
    Let $p_1$, $\overline{p}_1$, $p_2$, $\overline{p}_2$,
    be two pairs of complex conjugate points in $\C\PP^3$ that do not all lie in a line,
    and let $Q(x)$ be the $5$-dimensional linear system
    of all quadratic surfaces with basepoints at these four points.

    If the basepoints are not coplanar,
    then the \quadric{2}s in $Q(x)$ form three quadratic surfaces,
    $Q_i$, $Q_{s1}$, $Q_{s2}$,
    and four planes,
    $H_{p_1}$, $H_{\overline{p}_1}$, $H_{p_2}$, $H_{\overline{p}_2}$,
    where the real quadrics in $Q_i$ are in\-definite
    and the real quadrics in
    $Q_{s1}$, $Q_{s2}$, $H_{p_1}$, $H_{\overline{p}_1}$, $H_{p_2}$, $H_{\overline{p}_2}$
    are semidefinite.

    If the basepoints are coplanar,
    then the \quadric{2}s in $Q(x)$ form three quadratic surfaces,
    as in the nonplanar case,
    and in addition a web $W\mkern-4mu$,
    whose real quadrics are indefinite.
    In this case,
    the double plane containing the basepoints is a \quadric{1}
    that lies in the closure of each component of \quadric{2}s in $Q(x)$.
\end{lemma}

\begin{proof}
    First,
    note that the lines $\linspan{p_1, \overline{p}_1}$ and $\linspan{p_2, \overline{p}_2}$ are real and distinct,
    so if they intersect,
    they do so in a real point.

    The quadrics in $Q_i$ contain the two lines $\linspan{p_1, \overline{p}_1}$ and $\linspan{p_2, \overline{p}_2}$.
    Likewise,
    the quadrics in $Q_{s1}$
    contain the lines~$\linspan{p_1, p_2}$, $\linspan{\overline{p}_1, \overline{p}_2}$,
    and the quadrics in $Q_{s2}$
    contain the lines~$\linspan{p_1, \overline{p}_2}$, $\linspan{\overline{p}_1, p_2}$.
    The quadrics in $H_{p_1}$, $H_{\overline{p}_1}$, $H_{p_2}$, $H_{\overline{p}_2}$
    contain the plane~$\linspan{\overline{p}_1, p_2, \overline{p}_2}$,
    $\linspan{p_1, p_2, \overline{p}_2}$,
    $\linspan{p_1, \overline{p}_1, \overline{p}_2}$,
    $\linspan{p_1, \overline{p}_1, p_2}$, respectively.
    The first part of the claim follows from \cref{lem:line-definite}.

    Assume now that the basepoints span a plane $M$, which is real.
    Then $W$ consists of all quadrics $Q = M \cup N$, where $N$ is any plane.
    If $Q$ is real, then $N$ is also real,
    so $Q$ is indefinite when $N$ is distinct from $M$.
    On the other hand,
    the semidefinite double plane $2M$ is contained in $W, Q_i,Q_{s1}$ and $Q_{s2}$.
\end{proof}

\begin{remark}
    \label{rmk:not-relevant}
    If $V_\C(f_A)$ is a symmetroid of type~\ref{item:conic},
    then \cref{thm:symmsing} implies that $Q_A(x) \subset Q(x)$.
    A generic $3$-space in $Q(x)$
    intersects $H_{p_1}$, $H_{\overline{p}_1}$, $H_{p_2}$, $H_{\overline{p}_2}$
    in a point each outside of $Q_i$, $Q_{s1}$ and $Q_{s2}$.
    Since $Q_A(x)$ is a special $3$-space
    that intersects one of the quadratic surfaces~$Q_i$, $Q_{s1}$ or $Q_{s2}$
    in a conic section~$C$,
    it meets $H_{p_1}$, $H_{\overline{p}_1}$, $H_{p_2}$, $H_{\overline{p}_2}$
    in points on $C$ \cite[Proof of proposition~4.5]{He}.
    Hence $H_{p_1}$, $H_{\overline{p}_1}$, $H_{p_2}$, $H_{\overline{p}_2}$
    are not relevant in the analysis of the isolated nodes of $V_\C(f_A)$.
\end{remark}

\noindent
We now give a preliminary analysis of real singularities for spectrahedral symmetroids
with nonisolated singularities.

\begin{lemma}
    \label{lem:nonisolated}
    Let $S = V(f_A)$ be a rational quartic spectrahedral symmetroid with nonisolated singularities.
    Then $V_\C(f_A)$ has rank $2$ along a real line or a real conic,
    or it is singular and has rank $3$ along two intersecting complex conjugate lines.
    Furthermore:
    \begin{enumerate}
        \item
        A line of \point{2}s on $S$ is disjoint from the spectrahedron.

        \item
        A real conic of \point{2}s on $S$ may have no real points,
        or have a real point and be disjoint from the spectrahedron,
        or lie on the boundary of the spectrahedron.
    \end{enumerate}
\end{lemma}

\begin{proof}
    If the complex symmetroid $V_\C(f_A)$ is singular along a curve,
    then, by \cref{thm:symmsing},
    this curve contains a line or a smooth conic section.
    Furthermore, when $V_\C(f_A)$ is singular along a line,
    the matrix $A(x)$ may have rank $2$ or $3$ along the line. 

    In the first case,
    when $A(x)$ has rank $2$ along the line,
    the quadrics $Q_A(x)$ have four coplanar basepoints
    and the line is real.
    By \cref{lem:quadric-of-quadrics},
    the matrix $A(x)$ is indefinite along the line,
    so on the real spectrahedral symmetroid $V(f_A)$,
    the singular line must be disjoint from the spectrahedron.

    In the second case,
    when $A(x)$ has rank $3$ along the singular line,
    the web of quadrics $Q_A(x)$ contains a pencil $L\subset Q_A(x)$ of quadrics
    that are all singular at a basepoint cf.\ \cite[Proof of Proposition~3.5]{He}.
    Since this basepoint cannot be real,
    the complex conjugate is also a basepoint.
    But then,
    the complex conjugate pencil $\overline {L}\subset Q_A(x)$ must be distinct from $L$,
    and $V_\C(f_A)$ must be singular of rank $3$ along two complex conjugate lines.
    If these lines do not intersect,
    the symmetroid $V_\C(f_A)$ is a scroll of lines.
    The lines of this scroll form a curve of bidegree $(2,2)$
    on a quadratic surface in the Grassmannian of lines in $ \C\PP^3_y$,
    so the scroll is birational to an elliptic scroll, i.e., irrational.
    Therefore, the symmetroid $V_\C(f_A)$ that is singular,
    but of rank $3$ along two lines,
    is rational only if the two lines intersect.
    When the lines are complex conjugates,
    they of course intersect in a real point.
    Thus the real symmetroid $V(f_A)$ is singular at this point.

    If $V(f_A)$ is a rational spectrahedral symmetroid singular along a smooth conic section,
    then, by \cref{thm:symmsing},
    $A(x)$ must have rank $2$ along this curve
    and the web of quadrics $Q_A(x)$ has two pairs of complex conjugate basepoints
    that are linearly independent.
    Clearly the conic section is real and the \emph{a priori} listed possibilities follow.
\end{proof}

\noindent
We separate the proof of nonexistence of some cases
from the proof of \cref{thm:main}.
First we note that symmetroids of type~\ref{main-item:cyclide}
belong to a well-known class of real surfaces.
A real quartic surface singular along a conic section
with no real points in the plane at infinity,
is known as a \emph{cyclide} \cite[Chapter~V]{Jes}.

Not all cyclides are symmetroids.
By \cref{thm:symmsing}.\ref{item:conic},
a symmetroid singular along a conic has four additional nodes.
The cyclide $V\big(\big(x_0^2 + x_1^2 + x_2^2\big)^2 - x_3^4\big)$
has no singularities outside of the conic~$V\big(x_3, x_0^2 + x_1^2 + x_2^2\big)$,
hence it is not a symmetroid.
For cyclides with four additional nodes,
we have the following result:

\begin{proposition}[{\cite[Article~68]{Jes}}]
    If a cyclide has four additional nodes,
    then at most two of the isolated nodes are real.
\end{proposition}

\begin{corollary}
    A general, real, quartic symmetroid singular along a smooth conic section
    with no real points has either two or no real nodes.
\end{corollary}

\noindent
In a paper by Chandru, Dutta and Hoffmann,
the authors summarize classical works by Cayley \cite{Ca2} and Maxwell \cite{Max}.
This is used to produce a classification of the various forms of the cyclides \cite[Section~6]{CDH}.
Apart from the degenerated cases of a cone or a cylinder,
the cyclides are divided into three forms,
\emph{horned cyclides}, \emph{ring cyclides} and \emph{spindle cyclides}.
Of these,
only the ring cyclides have no real nodes.
The ring cyclides resemble squashed tori and do not bound a convex region.
Hence they do not occur as spectrahedral symmetroids.
For horned cyclides and spindle cyclides with precisely two real nodes,
the nodes connect two components of the real cyclide.
If these occur as spectrahedral symmetroids,
one of the components is the boundary of the spectrahedron.
We conclude:

\begin{proposition}
    \label{prop:cyclide}
    Let $S \coloneqq V(f_A)$ be a real quartic symmetroid with a nonempty spectrahedron
    that is singular along a real conic section with no real points.
    Then $S$ has two real nodes,
    both on the boundary of the spectrahedron.
\end{proposition}

We now disprove the existence of all spectrahedral symmetroids
of type~\ref{main-item:conic} with $(a, b) = (0, 0)$.

\begin{proposition}[\cite{Ste20}]
    \label{prop:conic-real-nodes}
    Let $S \coloneqq V(f_A) \subset \R\P^3$
    be a spectrahedral symmetroid,
    and suppose that the complex symmetroid~$V_\C(f_A) \subset \C\P^3$
    is of type~\ref{item:conic}.
    Then $S$ has a real, isolated node.
\end{proposition}

\begin{proof}
    The symmetroids with a given pair~$(a, b)$
    form a full-dimensional, Zariski open set
    in the families of symmetroids of types~\ref{main-item:conic-on-boundary},
    \ref{main-item:conic-disjoint} or \ref{main-item:cyclide}.
    Hence it suffices to show the nonexistence of symmetroids with
    $(a, b) = (0, 0)$ on a Zariski open set in each of the families.

    By \cref{thm:symmsing}.\ref{item:conic},
    $Q_A(x)$ has four linearly independent basepoints.
    Since $S$ is spectrahedral,
    the basepoints appear in complex conjugate pairs,
    $p_1$, $\overline{p}_1$ and $p_2$, $\overline{p}_2$.
    After a change of coordinates,
    we may assume that
    $p_1 \coloneqq [1 : i : 0 : 0]$ and $p_2 \coloneqq [0 : 0 : 1 : i]$.
    Consider the space~$Q(x)$ of all quadrics with~$p_1$,
    $\overline{p}_1$, $p_2$, $\overline{p}_2$ as basepoints.
    The quadrics in $Q(x)$ have matrices on the form
    \begin{equation*}
        \label{eq:matrix-general-basepoints}
        M(x)
        \coloneqq
        \begin{bmatrix}
            x_{00} &   0    & x_{02} & x_{03} \\
              0    & x_{00} & x_{12} & x_{13} \\
            x_{02} & x_{12} & x_{22} &   0    \\
            x_{03} & x_{13} &   0    & x_{22}
        \end{bmatrix}
        \mkern-7mu.
    \end{equation*}
    The quadratic surfaces described in \cref{lem:quadric-of-quadrics} are
    \begin{equation*}
        \begin{aligned}
            Q_i    &= V(x_{00}, x_{22}, x_{02}x_{13} - x_{03}x_{12}), \\
            Q_{s1} &= V(x_{02} - x_{13}, x_{03} + x_{12}, x_{00}x_{22} - x_{12}^2 - x_{13}^2),
            \\
            Q_{s2} &= V\big(x_{02} + x_{13}, x_{03} - x_{12}, x_{00}x_{22} - x_{12}^2 - x_{13}^2\big).
        \end{aligned}
    \end{equation*}
    By \cref{rmk:not-relevant},
    we do not have to consider $H_{p_1}$, $H_{\overline{p}_1}$, $H_{p_2}$, $H_{\overline{p}_2}$.

    For symmetroids of type~\ref{main-item:conic-on-boundary},
    $Q_A(x)$ intersects either $Q_{s1}$ or $Q_{s2}$ in a conic section;
    say $Q_{s1}$.
    Then the hyperplane spanned by $Q_{s1}$ and $Q_A(x)$
    equals $V(\lambda(x_{02} - x_{13}) + \mu(x_{03} + x_{12}))$
    for some constants $\lambda$, $\mu$.
    After conjugating $M(x)$ with the matrix
    \begin{equation*}
        \begin{bmatrix}
            \lambda & -\mu               & 0 & 0
            \\
            \mu     & \phantom{-}\lambda & 0 & 0
            \\
            0       & \phantom{-}0       & 1 & 0
            \\
            0       & \phantom{-}0       & 0 & 1
        \end{bmatrix}
        \mkern-7mu,
    \end{equation*}
    we may assume that $\mu = 0$.
    Thus $x_{02} = x_{13}$
    in the hyperplane $\linspan{Q_{s1}, Q_A(x)}$.
    If $Q_A(x)$ is sufficiently general,
    we may after a projective linear transformation
    assume that $A(x)$ is on the form
    \begin{equation}
        \label{eq:C1}
        \begin{bmatrix}
            x_0                               & 0   & x_1 &
            a_0x_0 + a_1x_1 + a_2x_2 + a_3x_3
            \\
            0                                 & x_0 & x_2 &
            x_1
            \\
            x_1                               & x_2 & x_3 &
            0
            \\
            a_0x_0 + a_1x_1 + a_2x_2 + a_3x_3 & x_1 & 0   &
            x_3
        \end{bmatrix}
        \mkern-7mu,
    \end{equation}
    for $a_0, a_1, a_2, a_3 \in \R$.
    The \locus{2} of $V_\C(f_A)$ is then
    \begin{align*}
        Q_A(x) \cap Q_{s1}
        &=
        V
        \big(
            a_0x_0 + a_1x_1 + (a_2 + 1)x_2 + a_3x_3,
            x_1^2 + x_2^2 - x_0x_3
        \big),
        \\
        Q_A(x) \cap Q_{s2}
        &=
        V
        \big(
            x_1,
            a_0x_0 + (a_2 - 1)x_2 + a_3x_3,
            x_2^2 - x_0x_3
       \big),
        \\
        Q_A(x) \cap Q_i
        &=
        V
        \big(
            x_0,
            x_3,
            x_1^2 - a_1x_1x_2 - a_2x_2^2
        \big).
    \end{align*}
    Assume for contradiction that $S$ has no real, isolated nodes.
    Then $Q_A(x) \cap Q_{s2}$ and $Q_A(x) \cap Q_i$ are not real.
    For $a_0 = 1$,
    this means that the discriminants
    \begin{align}
        \label{eq:C1-Ds2}
        D_{s2} &\coloneqq (a_2 - 1)^2 - 4a_3
        \\
        \label{eq:C1-Di}
        D_i &\coloneqq a_1^2 + 4a_2
    \end{align}
    are negative.
    Moreover,
    for $a_0 = 1$,
    Sylvester's criterion implies that the conic section~$Q_A(x) \cap Q_{s1}$
    is positive definite if and only if $D_{s2} + D_i < 0$.
    In other words,
    if $S$ has no real, isolated nodes,
    then $Q_A(x) \cap Q_{s1}$
    has no real points,
    so $S$ is not of type~\ref{main-item:conic-on-boundary}.

    \medskip
    \noindent
    For symmetroids of type~\ref{main-item:conic-disjoint},
    $Q_A(x)$ intersects $Q_i$ in a conic section.
    The hyperplane spanned by $Q_i$ and $Q_A(x)$
    equals $V(\lambda x_{00} + \mu x_{22})$
    for some constants $\lambda$, $\mu$.
    After rescaling $x_{22}$,
    we may assume that $x_{22} = \pm x_{00}$ in this hyperplane.
    If $x_{22} = -x_{00}$,
    then Sylvester's criterion implies that the hyperplane contains no definite matrix.
    Assume therefore that $x_{22} = x_{00}$.
    If $Q_A(x)$ is sufficiently general,
    we may after a projective linear transformation
    assume that $A(x)$
    is on the form
    \begin{equation}
        \label{eq:C2}
        \begin{bmatrix}
            x_0 &                                 0 &
            x_1 &                               x_2
            \\
              0 &                               x_0 &
            x_3 & a_0x_0 + a_1x_2 + a_2x_2 + a_3x_3
            \\
            x_1 &                               x_3 &
            x_0 &                                 0
            \\
            x_2 & a_0x_0 + a_1x_1 + a_2x_2 + a_3x_3 &
              0 &                               x_0
        \end{bmatrix}
    \end{equation}
    for $a_0, a_1, a_2, a_3 \in \R$.
    The isolated nodes of $V_\C(f_A)$ are then
    \begin{align*}
        Q_A(x) \cap Q_{s1}
        &=
        V
        \big(
            x_2 + x_3,
            a_0x_0 + (a_1 - 1)x_1 - (a_2 - a_3)x_3,
            x_0^2 - x_1^2 - x_3^2
        \big),
        \\
        Q_A(x) \cap Q_{s2}
        &=
        V
        \big(
            x_2 - x_3,
            a_0x_0 + (a_1 + 1)x_1 + (a_2 + a_3)x_3,
            x_0^2 - x_1^2 - x_3^2
        \big).
    \end{align*}
    For $a_0 = 1$,
    these are not real if the discriminants
    \begin{align}
        \label{eq:C2-Ds1}
        D_{s1}
        &\coloneqq
        a_1^2 + a_2^2 + a_3^2 - 2a_1 - 2a_2a_3
        \\
        \label{eq:C2-Ds2}
        D_{s2}
        &\coloneqq
        a_1^2 + a_2^2 + a_3^2 + 2a_1 + 2a_2a_3
    \end{align}
    are negative.
    But if $D_{s1} < 0$ and $D_{s2} < 0$,
    then $D_{s_1} + D_{s2} = 2(a_1^2 + a_2^2 + a_3^2) < 0$,
    which is impossible.
    Hence at least one of $Q_A(x) \cap Q_{s1}$ and $Q_A(x) \cap Q_{s2}$
    consists of real points.

    \medskip\noindent
    The case for symmetroids of type~\ref{main-item:cyclide}
    is covered by \cref{prop:cyclide}.
\end{proof}

\section{Proof of \texorpdfstring{\cref{thm:main}}{Theorem 1.7}}
\label{sec:proof}
\subsection{Type \texorpdfstring{\ref{main-item:triple}}{A}}

Triple points are \point{1}s,
hence the real ones are semidefinite
and therefore on the boundary of the spectrahedron.

By \cref{thm:symmsing}.\ref{item:triple},
a general complex symmetroid with a triple point has six nodes.
Since a spectrahedral symmetroid is a real surface,
the number $a$ of real nodes is even.
There are no further restraints on $b$,
the number of real semidefinite nodes,
as the examples in \cref{tab:A} show.

\subsection{Type \texorpdfstring{\ref{main-item:tacnode}}{B}}
\label{sec:proof-B}

By \cref{thm:symmsing}.\ref{item:tacnode},
the web $Q_A(x)$ has two basepoints.
Since $S$ is spectrahedral,
the basepoints are complex conjugates, $p$ and $\overline{p}$.
Let $Q(x)$ be the $7$-dimensional linear system of all quadratic surfaces with $p$ and $\overline{p}$ as basepoints.
The \locus{2} of $Q(x)$ consists of two fourfolds, $X_i$ and $X_s$.
The quadrics in $X_i$ are pairs of planes,
where one of the planes contains the line $\linspan{p, \overline{p}}$.
In $X_s$, the quadrics consist of two planes,
where the planes contain one basepoint each.
The set $\Sing(X_i) = \Sing(X_s)$ consists of pairs of planes that both contain $\linspan{p, \overline{p}}$.
The real quadrics in $X_i \setminus \Sing(X_i)$ are indefinite and the real quadrics in $X_s \setminus \Sing(X_s)$ are semidefinite.
The real quadrics in $\Sing(X_i) = \Sing(X_s)$ are either semidefinite or indefinite.

In the proof of \cite[Proposition~7.4]{He},
it is shown that the tacnode corresponds to a point in $\Sing(X_i) = \Sing(X_s)$,
and that $Q_A(x)$ intersects $X_i \setminus \Sing(X_i)$ in two points and $X_s \setminus \Sing(X_s)$ in four points.
The claim follows.

\subsection{Type \texorpdfstring{\ref{main-item:conic}}{C}}
\label{sec:proof-conic}

Let $C$ denote the conic section in the singular locus of $S$.
By \cref{thm:symmsing}.\ref{item:conic},
$S$ is the discriminant of a web $Q_A(x)$ of quadrics with four linearly independent basepoints.
In the notation of \cref{lem:quadric-of-quadrics},
$Q_A(x)$ intersects one of the surfaces $Q_i$, $Q_{s1}$ or $Q_{s2}$ in $C$
and the remaining two surfaces in two points each.
By \cref{rmk:not-relevant},
$S$ does not have any other \point{2}s.
Hence $a \le 4$.

Assume first that $C$ contains a real point.
If $C$ is disjoint from the spectrahedron, then $C \subset Q_i$.
The real quadrics in $Q_{s1}$ and $Q_{s2}$ are semidefinite, so $a = b$.
If $C$ lies on the boundary of the spectrahedron,
then $C$ is contained in either $Q_{s1}$ or $Q_{s2}$.
It follows that $a \le b + 2$,
since the real quadrics in $Q_i$ are indefinite.
We get $a \ge 2$ from \cref{prop:conic-real-nodes}.

The case where $C$ has no real points is dealt with in \cref{prop:cyclide}.

\subsection{Type \texorpdfstring{\ref{main-item:line}}{D}}

\cref{lem:nonisolated}
states that a line of \point{2}s is disjoint from the spectrahedron.

By \cref{thm:symmsing}.\ref{item:line},
$S$ is the discriminant of a web $Q_A(x)$ of quadrics with four coplanar basepoints.
Since $S$ has a nonempty spectrahedron, the basepoints occur in two complex conjugate pairs.
In the notation of \cref{lem:quadric-of-quadrics},
$Q_A(x)$ intersects $W$ in a line
and the surfaces $Q_{i}$, $Q_{s1}$, $Q_{s2}$ in two points each.
These are the only singularities on $S$, so we get that
\begin{equation*}
0 \le b \le a \le b + 2 \le 6,
\end{equation*}
and $a$, $b$ are even from \cref{lem:quadric-of-quadrics}.

\subsection{Type \texorpdfstring{\ref{main-item:two-lines}}{F}}

It follows from \cite[Remark~4.5]{Hel19}
that $S$ is the intersection of a $3$-space with $S'\subset \C\PP^5$,
where $S'$ is a quartic spectrahedral symmetroid that is singular of rank~$3$ along two complex conjugate $3$-spaces intersecting in a plane.
The claim is immediate from \cite[Proposition~4.6]{Hel19},
which says that $S'$ is singular along an additional quadratic surface whose real points lie on the spectrahedron.
\qed

\section{Deformation relations between rational symmetroids}
\label{sec:degenerations}

It is natural to ask whether the different types of rational symmetroids
listed in \cref{thm:main}
are deformations of each other.
That is,
does there exist a flat family of symmetroids
where one of the symmetroids is of one type,
while all the others are of a different type?
The following result addresses this:

\begin{proposition}
    Let $V(f_A) \subset \R\P^3$ be a quartic spectrahedral symmetroid
    and suppose that the complex symmetroid $V_\C(f_A) \subset \C\P^3$ is rational.
    If $V_\C(f_A)$ is of type \textup{\ref{item:tacnode}},
    then it can degenerate into a symmetroid of type~\textup{F}.
    This is the only instance of a symmetroid of type~\textup{\ref{item:triple}},
    \textup{\ref{item:tacnode}}, \textup{\ref{item:conic}}, \textup{\ref{item:line}}, or \hyperref[def:general-accidental-lines]{\textup{F}}
    that degenerates into one of the other types.
    Moreover,
    a real symmetroid $V(f_A)$
    does not degenerate from type~\ref{main-item:tacnode-on-boundary}
    to \ref{main-item:tacnode-disjoint};
    from \ref{main-item:conic-on-boundary} to
    \ref{main-item:conic-disjoint} or \ref{main-item:cyclide};
    or from \ref{main-item:two-lines-on-boundary} to \ref{main-item:two-lines-disjoint},
    or vice versa.
    In addition,
    a real symmetroid~$V(f_A)$ with $a$ real nodes, $b$ of which lie on the spectrahedron,
    does not degenerate into a symmetroid~$V(f_{A'})$ with $a'$ real nodes, $b'$ of which lie on the spectrahedron,
    if $(a, b) \neq (a', b')$.
\end{proposition}

\begin{proof}
    We start with the complex symmetroids.
    In \cite[Remark~9.2]{He},
    it is argued that if $S$ is a quartic surface
    singular along two intersecting lines, $L_1$ and $L_2$,
    with the implicit assumption that $L_1 \cap L_2$ is not a triple point,
    then $S$ satisfies the equation of a tacnodal surface.
    Hence a symmetroid of type~\ref{main-item:two-lines}
    can arise as a degeneration of surfaces
    of type \ref{item:tacnode} of \cref{thm:symmsing}.
    Consider the explicit deformation
    \begin{equation*}
        M(t)
        \coloneqq
        \begin{bmatrix}
            l_{00} & 0                           & l_{02}             &
            l_{03}
            \\
            0      & l_{00}                      & -l_{03} + tl_{12}  &
            l_{02} + tl_{13}
            \\
            l_{02} & -l_{03} + tl_{12}           & a_{22}x_0 + l_{22} &
            a_{23}x_0 + l_{23}
            \\
            l_{03} & \phantom{-}l_{02} + tl_{13} & a_{23}x_0 + l_{23} &
            a_{33}x_0 + l_{33}
        \end{bmatrix}
        \mkern-7mu,
    \end{equation*}
    where $l_{ij}$ are linear forms in $x_1$, $x_2$, $x_3$,
    and $a_{22}$, $a_{23}$, $a_{33}$, $t$ are scalars.
    For all $t$,
    the matrix $M(t)$ is on the form \eqref{eq:tacnode} of a tacnodal symmetroid.
    For $t = 0$,
    the matrix $M(0)$ takes the form \eqref{eq:matrix-two-lines}
    of a symmetroid singular along two lines.

    No other type of complex symmetroids is a degeneration of one of the other types.
    Most of the possible degenerations can be excluded by simple reasons,
    for instance a surface does not degenerate into one whose singular locus
    has lower degree or dimension,
    or into one with only singularities of lower multiplicity.
    We check the remaining cases:
    \begin{itemize}
        \item
        A symmetroid of type~\ref{item:triple}
        is not a degeneration of symmetroids of type~\ref{item:tacnode}.
        This is because a web of quadrics $Q_A(x)$ with basepoints
        does not degenerate into a web with no basepoints.
    
        \item
        Symmetroids of types~\ref{item:conic} or \ref{item:line}
        are not degenerations of symmetroids of type~\ref{item:tacnode}.
        Indeed,
        as noted in the proof of \cref{thm:main},
        the associated quadric at a tacnode
        is singular along the line spanned by two basepoints of $Q_A(x)$.
        In a degeneration of tacnodal symmetroids,
        one quadric in $Q_A(x)$ is singular along the line spanned by two basepoints.
        For symmetroids of types~\ref{item:conic} or \ref{item:line},
        $Q_A(x)$ has four basepoints.
        Assume that $Q_A(x)$ has four basepoints
        and recall the notation from \cref{lem:quadric-of-quadrics}.
        If the basepoints are independent
        and a quadric is singular along the line spanned by two of them,
        then it lies in one of the intersections
        $Q_i \cap Q_{s1}$, $Q_i \cap Q_{s2}$ or $Q_{s1} \cap Q_{s2}$.
        If the basepoints are coplanar
        and a quadric is singular along the line spanned by two of them,
        then it lies in one of the intersections
        $W \cap Q_i$, $W \cap Q_{s1}$ or $W \cap Q_{s2}$.
        In either case,
        this implies that the surface has one fewer isolated \point{2},
        than one of type~\ref{item:conic} or \ref{item:line}.

        \item
        A symmetroid of type~\ref{item:conic}
        is not a degeneration of symmetroids of type~\ref{item:line}.
        This is because a web of quadrics~$Q_A(x)$ with linearly dependent basepoints
        does not degenerate into one with independent basepoints.

        \item
        A symmetroid of type~\hyperref[def:general-accidental-lines]{\textup{F}}
        is not a degeneration 
        of symmetroids of types~\ref{item:conic} or \ref{item:line}.
        This is because a symmetroid with a curve in its \locus{2}
        does not degenerate into one without a curve in its \locus{2}.
    \end{itemize}
    For the real symmetroids,
    we note that a positive semidefinite \rankmatrix{2} has two positive eigenvalues,
    while an indefinite \rankmatrix{2} has eigenvalues with different signs.
    In a degeneration from one into the other,
    one of the eigenvalues changes to $0$,
    causing the rank to drop.
    This implies that a symmetroid with singularities on the spectrahedron
    does not degenerate into one with singularities outside of the spectrahedron,
    and vice versa.

    Finally,
    in a degeneration of a surface with two complex conjugate nodes
    into one with two real nodes,
    the imaginary part of the coordinates of the nodes changes to $0$,
    while the coordinates have the same real part for both nodes.
    Hence the two nodes coincide.
    Thus a symmetroid with $a$ real nodes
    does not degenerate into one with $a' \neq a$ real nodes.
\end{proof}

\chapter{Examples of existence}
\label[section]{app:examples}

This section lists matrices that define spectrahedral symmetroids
with the different values of $(a, b)$ described by \cref{thm:main}.
There is one missing instance,
namely \ref{main-item:line} with $(a, b) = (0, 0)$.
For brevity in the tables,
we write $(a_{00}, a_{01}, a_{02}, a_{03}, a_{11}, a_{12}, a_{13}, a_{22}, a_{23}, a_{33})$
for the symmetric matrix
\begin{equation*}
    \begin{bmatrix}
        a_{00} & a_{01} & a_{02} & a_{03}
        \\
        a_{01} & a_{11} & a_{12} & a_{13}
        \\
        a_{02} & a_{12} & a_{22} & a_{23}
        \\
        a_{03} & a_{13} & a_{23} & a_{33}
    \end{bmatrix}
    \mkern-7mu.
\end{equation*}
\cref{tab:A,tab:B,tab:F} provide one such $10$-tuple
for each of the matrices $A_2$ and $A_3$ in \eqref{eq:AxAx}.
The matrices $A_0$ and $A_1$ are fixed,
and therefore not listed.

\section{Spectrahedral symmetroids with a triple point}
\label[section]{app:triple}

Triple points correspond precisely to \point{1}s.
Thus \eqref{eq:AxAx} defines a spectrahedral symmetroid with a triple point
if $A_0$ is a positive definite matrix, $A_1$ is a \rankmatrix{1},
and $A_2$, $A_3$ are any symmetric matrices.
To find examples,
we let $A_0$ be the identity matrix~$I_4$,
\begin{equation*}
    A_1
    \coloneqq
    \begin{bmatrix}
        1 & 0 & 0 & 0 \\
        0 & 0 & 0 & 0 \\
        0 & 0 & 0 & 0 \\
        0 & 0 & 0 & 0
    \end{bmatrix}
    \mkern-7mu,
\end{equation*}
and used a pseudorandom generator to draw symmetric matrices $A_2$ and $A_3$ with integer entries between $-9$ and $9$.
Doing this,
we found all values of $(a, b)$ specified by \cref{thm:main}.\ref{main-item:triple}.
The results are listed in \cref{tab:A}.

\begin{longtable}{@{}
        m{0.08\textwidth}
        >{\(}m{0.42\textwidth}<{\)}
        >{\(}m{0.42\textwidth}<{\)}
        @{}}
    \caption{Examples of matrices defining spectrahedral symmetroids
        of type~\ref{main-item:triple}.}
    \label{tab:A}
    \\*
    \toprule
    $\boldsymbol{(a, b)}$
    &
    \boldsymbol{A_2}
    &
    \boldsymbol{A_3}
    \\*
    \midrule
    \endfirsthead
    \toprule
    $\boldsymbol{(a, b)}$
    &
    \boldsymbol{A_2}
    &
    \boldsymbol{A_3}
    \\*
    \midrule
    \endhead
    \addlinespace
    \multicolumn{3}{@{}l}{Continued on the next page.}
    \endfoot
    \bottomrule
    \endlastfoot
    
    $(6, 6)$
    &
    (-9, -1, 2, -2, -3, 7, 7, 5, 5, -4)
    &
    (3, -4, -7, 3, -1, -2, -9, 4, 8, 6)
    \\*
    \midrule
    
    $(6, 5)$
    &
    (3, -7, 5, 4, 3, 2, 9, 5, 7, -8)
    &
    (-3, 5, 5, -5, -3, -6, 3, -2, 0, -7)
    \\*
    \midrule
    
    $(6, 4)$
    &
    (-7, 7, -6, 6, 5, 5, 5, -8, -2, 1)
    &
    (-3, 8, -2, 1, -4, 1, 5, -8, 8, -7)
    \\*
    \midrule
    
    $(6, 3)$
    &
    (0, 0, 3, -3, 0, -2, 8, 6, 7, 1)
    &
    (-2, 1, 5, 4, 5, 6, -8, 5, 1, 6)
    \\*
    \midrule
    
    $(6, 2)$
    &
    (9, 7, 2, 3, -5, -9, -2, -2, 3, -5)
    &
    (-9, 6, -3, 3, -7, 2, -1, 2, -7, 5)
    \\*
    \midrule
    
    $(6, 1)$
    &
    (8, -5, 2, -9, 1, -1, 2, -5, 9, -8)
    &
    (-9, 6, -3, 3, -7, 2, -1, 2, -7, 5)
    \\*
    \midrule
    
    $(6, 0)$
    &
    (-3, 6, -4, 1, 2, 6, 9, 0, -7, 8)
    &
    (8, 6, 3, -4, 5, 9, 7, 3, 7, -9)
    \\*
    \midrule
    
    $(4, 4)$
    &
    (1, -6, -6, 4, 6, 2, 5, -5, -1, -8)
    &
    (-5, 1, -7, 6, 9, 9, 7, -9, -8, -2)
    \\*
    \midrule
    
    $(4, 3)$
    &
    (6, 3, 9, 9, -8, 9, 0, -7, 6, -7)
    &
    (-1, 3, 3, 3, -9, 5, -6, 5, 4, -9)
    \\*
    \midrule
    
    $(4, 2)$
    &
    (6, 8, -3, 9, 2, -2, -9, 4, 6, 7)
    &
    (-2, 9, -4, -2, 8, -1, 9, 1, -4, 4)
    \\*
    \midrule
    
    $(4, 1)$
    &
    (2, 9, -1, -8, 1, 0, -1, -8, 6, -2)
    &
    (2, -6, 0, -6, -5, 2, -1, 6, -1, 9)
    \\*
    \midrule
    
    $(4, 0)$
    &
    (-8, 0, -9, 6, 3, -6, 3, -7, 6, 6)
    &
    (1, -5, -1, 8, -9, 0, 4, -2, 8, -3)
    \\*
    \midrule
    
    $(2, 2)$
    &
    (4, -4, 1, 6, 3, 2, 0, -5, 1, -3)
    &
    (4, 3, 9, 5, -6, -3, -5, 2, 0, 4)
    \\*
    \midrule
    
    $(2, 1)$
    &
    (-5, 8, -1, -6, 8, 7, 9, 5, 9, 6)
    &
    (3, 1, 4, 5, -4, -6, -8, -2, 8, 5)
    \\*
    \midrule
    
    $(2, 0)$
    &
    (-9, -6, 5, -1, 4, 0, 1, 8, 1, 6)
    &
    (-7, 5, -7, 4, 5, 0, 3, 6, 4, 7)
    \\*
    \midrule
    
    $(0, 0)$
    &
    (-5, 9, 5, 1, 1, -2, 9, -2, 0, 2)
    &
    (7, -6, 5, 1, -4, 1, 3, -9, 8, -5)
\end{longtable}

\section{Spectrahedral symmetroids with a tacnode}

Consider the space~$Q(x)$ of all quadrics with basepoints~$p$, $\overline{p}$.
After a change of coordinates,
we may assume that $p \coloneqq [1 : i : 0 : 0]$.
Then quadrics in $Q(x)$ have matrices on the form
\begin{equation}
    \label{eq:two-basepoints-matrix}
    \begin{bmatrix}
        x_{00} &    0   & x_{02} & x_{03}
        \\
           0   & x_{00} & x_{12} & x_{13}
        \\
        x_{02} & x_{12} & x_{22} & x_{23}
        \\
        x_{03} & x_{13} & x_{23} & x_{33}
    \end{bmatrix}
    \mkern-7mu.
\end{equation}
Using the notation from \cref{sec:proof-B},
a $3$-space $Q_A(x) \subset Q(x)$,
corresponding to a tacnodal symmetroid,
intersects $\Sing(X_i) = V(x_{00}, x_{02}, x_{03}, x_{12}, x_{13})$
in a point.
Hence $A(x)$ has the form 
\begin{equation}
    \label{eq:tacnode}
    A(x)
    \coloneqq
    \begin{bmatrix}
        l_{00} & 0      & l_{02}             & l_{03}
        \\
        0      & l_{00} & l_{12}             & l_{13}
        \\
        l_{02} & l_{12} & a_{22}x_0 + l_{22} & a_{23}x_0 + l_{23}
        \\
        l_{03} & l_{13} & a_{23}x_0 + l_{23} & a_{33}x_0 + l_{33}
    \end{bmatrix}
    \mkern-7mu,
\end{equation}
where each $l_{ij}$ is a linear form in $x_1$, $x_2$, $x_3$,
and $a_{ij} \in \R$.
Moreover,
if $V(f_A)$ is spectrahedral,
we can take $A_1 = A([0 : 1 : 0 : 0])$ to be positive definite.

To find examples,
we let $A_1$ in \eqref{eq:AxAx} be the identity matrix~$I_4$,
and
\begin{equation*}
    A_0
    \coloneqq
    \begin{bmatrix}
        0 & 0 & 0      & 0
        \\
        0 & 0 & 0      & 0
        \\
        0 & 0 & \alpha & 0
        \\
        0 & 0 & 0      & \beta
    \end{bmatrix}
    \mkern-7mu.
\end{equation*}
For symmetroids with a tacnode on the boundary of the spectrahedron,
we chose $\alpha = 1$ and $\beta = 2$.
Likewise,
we chose $\alpha = 1$ and $\beta =-2$
for symmetroids with a tacnode disjoint from the spectrahedron.
We used a pseudorandom generator to draw symmetric matrices~$A_2$ and $A_3$
on the form~\eqref{eq:two-basepoints-matrix} with integer entries between $-9$ and $9$.
Doing this,
we found all values of $(a, b)$ specified
by \cref{thm:main}.\ref{main-item:tacnode-on-boundary}
and \cref{thm:main}.\ref{main-item:tacnode-disjoint}.
The results are listed in \cref{tab:B}.

\begin{longtable}{@{}
        m{0.08\textwidth}
        >{\(}m{0.39\textwidth}<{\)}
        >{\(}m{0.38\textwidth}<{\)}
        @{}}
    \caption{Examples of matrices defining spectrahedral symmetroids
        of types~\ref{main-item:tacnode-on-boundary}
        and \ref{main-item:tacnode-disjoint}.}
    \label{tab:B}
    \\*
    \toprule
    $\boldsymbol{(a, b)}$
    &
    \textbf{\ref{main-item:tacnode-on-boundary}}
    &
    \textbf{\ref{main-item:tacnode-disjoint}}
    \\*
    \midrule
    \endfirsthead
    \toprule
    $\boldsymbol{(a, b)}$
    &
    \textbf{\ref{main-item:tacnode-on-boundary}}
    &
    \textbf{\ref{main-item:tacnode-disjoint}}
    \\*
    \midrule
    \endhead
    \addlinespace
    \multicolumn{3}{@{}l}{Continued on the next page.}
    \endfoot
    \bottomrule
    \endlastfoot

    $(6, 4)$
    &
    (6, 0, -3, 6, 6, -6, -4, 0, 3, 6)
    \newline
    (3, 0, 5, -8, 3, -3, -4, 2, -5, 8)
    &
    (8, 0, -6, -8, 8, 7, 1, -7, 4, 7)
    \newline
    (6, 0, 2, -2, 6, -7, 6, -1, 9, 7)
    \\*
    \midrule

    $(4, 4)$
    &
    (0, 0, 4, -4, 0, 3, 3, 1, -8, 4)
    \newline
    (5, 0, -2, 4, 5, -5, 2, 1, -5, -8)
    &
    (4, 0, -8, -2, 4, 5, -3, 5, -8, 9)
    \newline
    (3, 0, 2, -4, 3, 0, 4, 0, 2, 6)
    \\*
    \midrule

    $(4, 2)$
    &
    (8, 0, -5, 8, 8, 8, -3, 5, -3, 2)
    \newline
    (-8, 0, 1, 9, -8, -9, 7, -8, 8, 9)
    &
    (1, 0, 5, -7, 1, 6, 4, -6, -3, 6)
    \newline
    (6, 0, 2, -8, 6, 1, -3, 5, 4, -8)
    \\*
    \midrule

    $(2, 2)$
    &
    (6, 0, -8, -4, 6, 1, 9, -8, 0, 3)
    \newline
    (3, 0, -2, 7, 3, 7, 6, -6, 4, 1)
    &
    (7, 0, 6, 5, 7, 2, -4, 0, -7, 2)
    \newline
    (0, 0, -8, 2, 0, 3, 7, 8, -3, 7)
    \\*
    \midrule

    $(2, 0)$
    &
    (6, 0, 6, -5, 6, 6, -1, 1, -1, 7)
    \newline
    (8, 0, 5, -8, 8, -6, 1, 2, -2, -9)
    &
    (5, 0, 6, -6, 5, -6, 8, -8, 7, 6)
    \newline
    (4, 0, 2, -1, 4, 6, -6, -2, 3, 5)
    \\*
    \midrule

    $(0, 0)$
    &
    (-1, 0, 2, 4, -1, 0, 1, 7, 5, 5)
    \newline
    (7, 0, 4, -6, 7, 5, 8, -2, 0, -4)
    &
    (9, 0, 3, -9, 9, 6, -7, -4, 1, 1)
    \newline
    (4, 0, 5, -9, 4, -2, 5, 3, 1, -4)
\end{longtable}

\section{Spectrahedral symmetroids with a double conic}
\label[section]{app:conic}

To find examples of symmetroids~$V(f_A)$ of type~\ref{main-item:conic-on-boundary},
we can take $A(x)$ to be on the form~\eqref{eq:C1}.
For $a_0 = 1$,
we get
$(a, b) = (4, 2$) if both discriminants~\eqref{eq:C1-Ds2} and \eqref{eq:C1-Di} are positive,
$(a, b) = (2, 2)$ if \eqref{eq:C1-Ds2} is positive and \eqref{eq:C1-Di} is negative,
and
$(a, b) = (2, 0)$ if \eqref{eq:C1-Ds2} is negative and \eqref{eq:C1-Di} is positive.
It remains to check that $A(x)$ contains a definite matrix
to conclude that $V(f_A)$ is in fact spectrahedral.
In particular,
the following examples are spectrahedral:
\begin{itemize}
    \item
    $a_0 \coloneqq 1$, $a_1 \coloneqq 0$, $a_2 \coloneqq 1$, $a_3 \coloneqq -1$
    gives a symmetroid with $(a, b) = (4, 2)$;

    \item
    $a_0 \coloneqq 1$, $a_1 \coloneqq 0$, $a_2 \coloneqq -2$, $a_3 \coloneqq -4$
    gives a symmetroid with $(a, b) = (2, 2)$;
        
    \item
    $a_0 \coloneqq 1$, $a_1 \coloneqq -3$, $a_2 \coloneqq 0$, $a_3 \coloneqq 1$
    gives a symmetroid with $(a, b) = (2, 0)$.
\end{itemize}

To find examples of symmetroids~$V(f_A)$ of type~\ref{main-item:conic-disjoint},
we can take $A(x)$ to be on the form \eqref{eq:C2}.
For $a_0 = 1$,
we get $(a, b) = (4, 4)$ if both discriminants~\eqref{eq:C2-Ds1} and \eqref{eq:C2-Ds2} are positive,
and $(a, b) = (2, 2)$ if only one of \eqref{eq:C2-Ds1} and \eqref{eq:C2-Ds2} is positive.
It remains to check that $A(x)$ contains a definite matrix
to conclude that $V(f_A)$ is in fact spectrahedral.
In particular,
the following examples are spectrahedral:
\begin{itemize}
    \item
    $a_0 \coloneqq 1$, $a_1 \coloneqq 3$, $a_2 \coloneqq 0$, $a_3 \coloneqq 0$
    gives a symmetroid with $(a, b) = (4, 4)$;

    \item
    $a_0 \coloneqq 1$, $a_1 \coloneqq \frac{1}{2}$, $a_2 \coloneqq 0$, $a_3 \coloneqq 0$
    gives a symmetroid with $(a, b) = (2, 2)$.
\end{itemize}

To find examples of symmetroids~$V(f_A)$ of type~\ref{main-item:cyclide},
we can again take $A(x)$ to be on the form \eqref{eq:C1}.
For $a_0 = 1$,
we get the correct type
if the discriminant~$D_{s2}$ from \eqref{eq:C1-Ds2} is positive,
$D_i$ from \eqref{eq:C1-Di} is negative
and $D_{s2} < -D_i$.
It remains to check that $A(x)$ contains a definite matrix
to conclude that $V(f_A)$ is in fact spectrahedral.
In particular, \eqref{eq:C1} gives a spectrahedral symmetroid
of type~\ref{main-item:cyclide} for $a_0 \coloneqq 1$, $a_1 \coloneqq 2$, $a_2 \coloneqq -5$ and $a_3 \coloneqq 6$.

\section{Spectrahedral symmetroids with rank 2 along a double line}
\label[section]{app:line}

Consider the space~$Q(x)$ of quadrics with coplanar basepoints~$p_1$,
$\overline{p}_1$, $p_2$, $\overline{p}_2$.
After a change of coordinates,
we may assume that
$p_1 \coloneqq [1 : i : 0 : 0]$ and $p_2 \coloneqq [1 : 0 : i : 0]$.
Then the quadrics in $Q(x)$ have matrices on the form
\begin{equation*}
    \label{eq:matrix-coplanar-basepoints}
    \begin{bmatrix}
        x_{00} &   0    &   0    & x_{03} \\
        0      & x_{00} & x_{12} & x_{13} \\
        0      & x_{12} & x_{00} & x_{23} \\
        x_{03} & x_{13} & x_{23} & x_{33}
    \end{bmatrix}
    \mkern-7mu.
\end{equation*}
If $Q_A(x) \subset Q(x)$ is a generic $3$-space,
we may after a projective linear transformation
assume that $A(x)$ is on the form
\begin{equation}
    \label{eq:line-matrix}
    \begin{bmatrix}
        x_0 & 0                           &
        0   & x_1
        \\
        0   & x_0                         &
        x_2 & x_3
        \\
        0   & x_2                         &
        x_0 & a_0x_0 + a_1x_1 + a_2x_2 + a_3x_3
        \\
        x_1 & x_3                         &
        a_0x_0 + a_1x_1 + a_2x_2 + a_3x_3 &
        b_0x_0 + b_1x_1 + b_2x_2 + b_3x_3
    \end{bmatrix}
    \mkern-7mu,
\end{equation}
for $a_i, b_i \in \R$.
Using the notation from \cref{lem:quadric-of-quadrics},
the \locus{2} of $Q_A(x)$ is
\begin{align*}
    Q_A(x) \cap W
    &=
    V(x_0, x_2),
    \\
    Q_A(x) \cap Q_i
    &=
    V\big(x_0, x_1, b_2x_2^2 + (-2a_2 + b_3)x_2x_3 - 2a_3x_3^2\big),
    \\
    Q_A(x) \cap Q_{s1}
    &=
    V
    \big(
        x_0 - x_2,
        a_1x_1 + (a_0 + a_2)x_2 + (a_3 - 1)x_3,
        \\
        &\qquad\qquad
        x_1^2 - b_1x_1x_2 - (b_0 + b_2)x_2^2 - b_3x_2x_3 + x_3^2
    \big),
    \\
    Q_A(x) \cap Q_{s2}
    &=
    V
    \big(
        x_0 + x_2,
        a_1x_1 - (a_0 - a_2)x_2 + (a_3 + 1)x_3,
        \\
        &\qquad\qquad
        x_1^2 + b_1x_1x_2 - (b_0 - b_2)x_2^2 + b_3x_2x_3 + x_3^2
    \big).
\end{align*}
Hence the matrix $A(x)$ defines a symmetroid~$V(f_A)$
of type~\ref{main-item:line}.

For $a_1 = 1$,
the reality of the isolated nodes is determined by the discriminants
\begin{align*}
    D_i
    &\coloneqq
    (2a_2 + b_3)^2 + 8b_2a_3,
    \\
    D_{s1}
    &\coloneqq
    (a_3(2a_0 + 2a_2 + b_1) - 2a_0 - 2a_2 - b_1 - b_3)^2
    \\
    &\qquad
    - 4(a_3^2 - 2a_3 + 2)((a_0 + a_2)(a_0 + a_2 + b_1) - b_0 - b_2),
    \\
    D_{s2}
    &\coloneqq
    (a_3(-2a_0 + 2a_2 - b_1) - 2a_0 + 2a_2 - b_1 + b_3)^2
    \\
    &\qquad
    - 4(a_3^2 + 2a_3 + 2)((a_0 - a_2)(a_0 - a_2 + b_1) - b_0 + b_2).
\end{align*}
More precisely,
\begin{itemize}
    \item
    $(a, b) = (6, 4)$ if $D_i$, $D_{s1}$ and $D_{s2}$ are positive;

    \item
    $(a, b) = (4, 4)$ if $D_{s1}$ and $D_{s2}$ are positive and $D_{i}$ is negative;

    \item
    $(a, b) = (4, 2)$ if $D_i$ is positive and either $D_{s1}$ or $D_{s2}$ is positive;

    \item
    $(a, b) = (2, 2)$ if either $D_{s1}$ or $D_{s2}$ is positive and $D_i$ is negative,

    \item
    $(a, b) = (2, 0)$ if $D_i$ is positive and $D_{s1}$ and $D_{s2}$ are negative;

    \item
    $(a, b) = (0, 0)$ if $D_i$, $D_{s1}$, $D_{s2}$ are negative.
\end{itemize}
It remains to check that $A(x)$ contains a definite matrix
to conclude that $V(f_A)$ is in fact spectrahedral.
Spectrahedral examples are given in \cref{tab:D},
except $(a, b) = (0, 0)$.

\begin{table}[htbp]
    \centering
    \caption{Examples of parameters for \eqref{eq:line-matrix}
        that define spectrahedral symmetroids of type~\ref{main-item:line}.}
    \label{tab:D}
    \begin{tabular}{@{}>{\(}c<{\)}>{\(}c<{\)}>{\(}c<{\)}
                       >{\(}r<{\)}>{\(}r<{\)}>{\(}c<{\)}
                       >{\(}c<{\)}>{\(}r<{\)}>{\(}c<{\)}@{}}
        \toprule
        \boldsymbol{(a, b)} & \boldsymbol{a_0} & \boldsymbol{a_1} &
        \boldsymbol{a_2}    & \boldsymbol{a_3} & \boldsymbol{b_0} &
        \boldsymbol{b_1}    & \boldsymbol{b_2} & \boldsymbol{b_3}
        \\
        \midrule
        (6, 4) & 0 & 1 &  0 &  1 & 0 & 0 &  0 & 1
        \\
        (4, 4) & 0 & 1 &  0 & -1 & 1 & 0 &  1 & 1
        \\
        (4, 2) & 0 & 1 & -1 &  2 & 0 & 0 &  0 & 1
        \\
        (2, 2) & 0 & 1 &  0 &  1 & 0 & 0 & -1 & 1
        \\
        (2, 0) & 0 & 1 & -1 &  1 & 0 & 0 &  0 & 1
        \\
        \bottomrule
    \end{tabular}
\end{table}

\begin{remark}
    If we let
    $a_0 \coloneqq 3$, $a_1 \coloneqq 1$, $a_2 \coloneqq 0$,
    $a_3 \coloneqq 1$, $b_0 \coloneqq 0$, $b_1 \coloneqq 0$,
    $b_2 \coloneqq -1$, $b_3 \coloneqq 0$
    then \eqref{eq:line-matrix} defines a symmetroid
    with rank $2$ along a line and no real, isolated nodes.
    By Sylvester's criterion,
    it is not spectrahedral.
    We have not been able to find a spectrahedral symmetroid
    of type~\ref{main-item:line} with $(a, b) = (0, 0)$,
    nor prove its nonexistence.
\end{remark}

\section{Spectrahedral symmetroids with rank 3 along two double lines}
\label[section]{app:two-lines}

After a change of coordinates,
\cref{prop:general-accidental-lines}
implies that if $V_C(f_A)$ is of type \ref{main-item:two-lines},
then $Q_A(x)$ is contained in the $7$-space $Q(x)$ defined by \eqref{eq:two-basepoints-matrix}.
The discriminant of $Q(x)$ is singular of rank $3$
along the complex conjugate $4$-spaces
$H_4 \coloneqq V(x_{00}, x_{02} - i x_{12}, x_{03} - i x_{13})$
and $\overline{H}_4 \coloneqq V(x_{00}, x_{02} + i x_{12}, x_{03} + i x_{13})$.
In order to find examples of $3$-spaces $Q_A(x)$
corresponding to symmetroids of type~\ref{main-item:two-lines},
we consider the $5$-space $H \coloneqq V(x_{02} - x_{03}, x_{03} + x_{13})$.
Because $H$ intersects $H_4$ and $\overline{H}_4$ in a $3$-space each,
a generic $3$-space $Q_A(x) \subset H$
corresponds to a symmetroid of type~\ref{main-item:two-lines}.
We therefore consider matrices on the form
\begin{equation*}
    A(x)
    \coloneqq
    \begin{bmatrix*}[r]
        l_{00} & 0\,\,\, &  l_{02} & l_{03}
        \\
        0\,\,\ &  l_{00} & -l_{03} & l_{02}
        \\
        l_{02} & -l_{03} &  l_{22} & l_{23}
        \\
        l_{03} &  l_{02} &  l_{23} & l_{33}
    \end{bmatrix*}
    \mkern-7mu,
\end{equation*}
where $l_{ij}$ are linear forms in $x_0$, $x_1$, $x_2$, $x_3$.
In order to make $A(x)$ resemble \eqref{eq:tacnode},
we let $A_0 = A([1 : 0 : 0 : 0])$ correspond to the intersection
between the lines in the singular locus of $V(f_A)$.
Hence we reduce to
\begin{equation}
    \label{eq:matrix-two-lines}
    A(x)
    \coloneqq
    \begin{bmatrix*}
        l_{00} & 0       &  l_{02}             & l_{03}
        \\
        0      &  l_{00} & -l_{03}             & l_{02}
        \\
        l_{02} & -l_{03} &  a_{22}x_0 + l_{22} & a_{23}x_0 + l_{23}
        \\
        l_{03} &  l_{02} &  a_{23}x_0 + l_{23} & a_{33}x_0 + l_{33}
    \end{bmatrix*}
    \mkern-7mu,
\end{equation}
where $l_{ij}$ are linear forms in $x_1$, $x_2$, $x_3$.
Furthermore,
if $V(f_A)$ is spectrahedral,
we can take $A_1 = A([0 : 1 : 0 : 0])$ to be a definite matrix.

We chose $A_1$ to be the positive definite matrix
\begin{equation*}
    \begin{bmatrix}
        1 & 0 & 0 & 0 \\
        0 & 1 & 0 & 0 \\
        0 & 0 & 1 & 1 \\
        0 & 0 & 1 & 2
    \end{bmatrix}
    \mkern-7mu.
\end{equation*}
For type~\ref{main-item:two-lines-on-boundary},
we let $a_{22} \coloneqq 1$, $a_{23} \coloneqq 1$, $a_{33} \coloneqq 3$ in $A_0$.
For type~\ref{main-item:two-lines-disjoint},
we let $a_{22} \coloneqq 1$, $a_{23} \coloneqq 1$ and $a_{33} \coloneqq 0$.
\cref{tab:F}  shows $A_2$ and $A_3$ realizing all values of $(a, b)$
specified by \cref{thm:main}.\ref{main-item:two-lines-on-boundary}
and \cref{thm:main}.\ref{main-item:two-lines-disjoint}.

\begin{longtable}{@{}
                  m{0.08\textwidth}
                  >{\(}m{0.33\textwidth}<{\)}
                  >{\(}m{0.31\textwidth}<{\)}
                  @{}}
    \caption{Examples of matrices defining spectrahedral symmetroids
            of types~\ref{main-item:two-lines-on-boundary}
            and \ref{main-item:two-lines-disjoint}.}
    \label{tab:F}
    \\*
    \toprule
    $\boldsymbol{(a, b)}$
    &
    \textbf{\ref{main-item:two-lines-on-boundary}}
    &
    \textbf{\ref{main-item:two-lines-disjoint}}
    \\*
    \midrule
    \endfirsthead
    \toprule
    $\boldsymbol{(a, b)}$
    &
    \textbf{\ref{main-item:two-lines-on-boundary}}
    &
    \textbf{\ref{main-item:two-lines-disjoint}}
    \\*
    \midrule
    \endhead
    \addlinespace
    \multicolumn{3}{@{}l}{Continued on the next page.}
    \endfoot
    \bottomrule
    \endlastfoot

    $(2, 2)$
    &
    (0, 0, 1, 0, 0, 0, 1, 0, 1, 0)
    \newline
    (0, 0, 0, 1, 0, -1, 0, 0, 0, 0)
    &
    (0, 0, 1, 0, 0, 0, 1, 0, 1, 0)
    \newline
    (0, 0, 0, 1, 0, -1, 0, 0, 0, 3)
    \\*
    \midrule

    $(0, 0)$
    &
    (0, 0, 1, 0, 0, 0, 1, 1, 1, 0)
    \newline
    (0, 0, 0, 1, 0, -1, 0, 0, 0, 0)
    &
    (0, 0, 1, 0, 0, 0, 1, 0, 1, 0)
    \newline
    (0, 0, 0, 1, 0, -1, 0, 0, 0, 0)
\end{longtable}

\clearpage
\section{Pictures}

\begin{figure}[!htbp]
    \centering
    \includegraphics[height = 0.18\textheight]{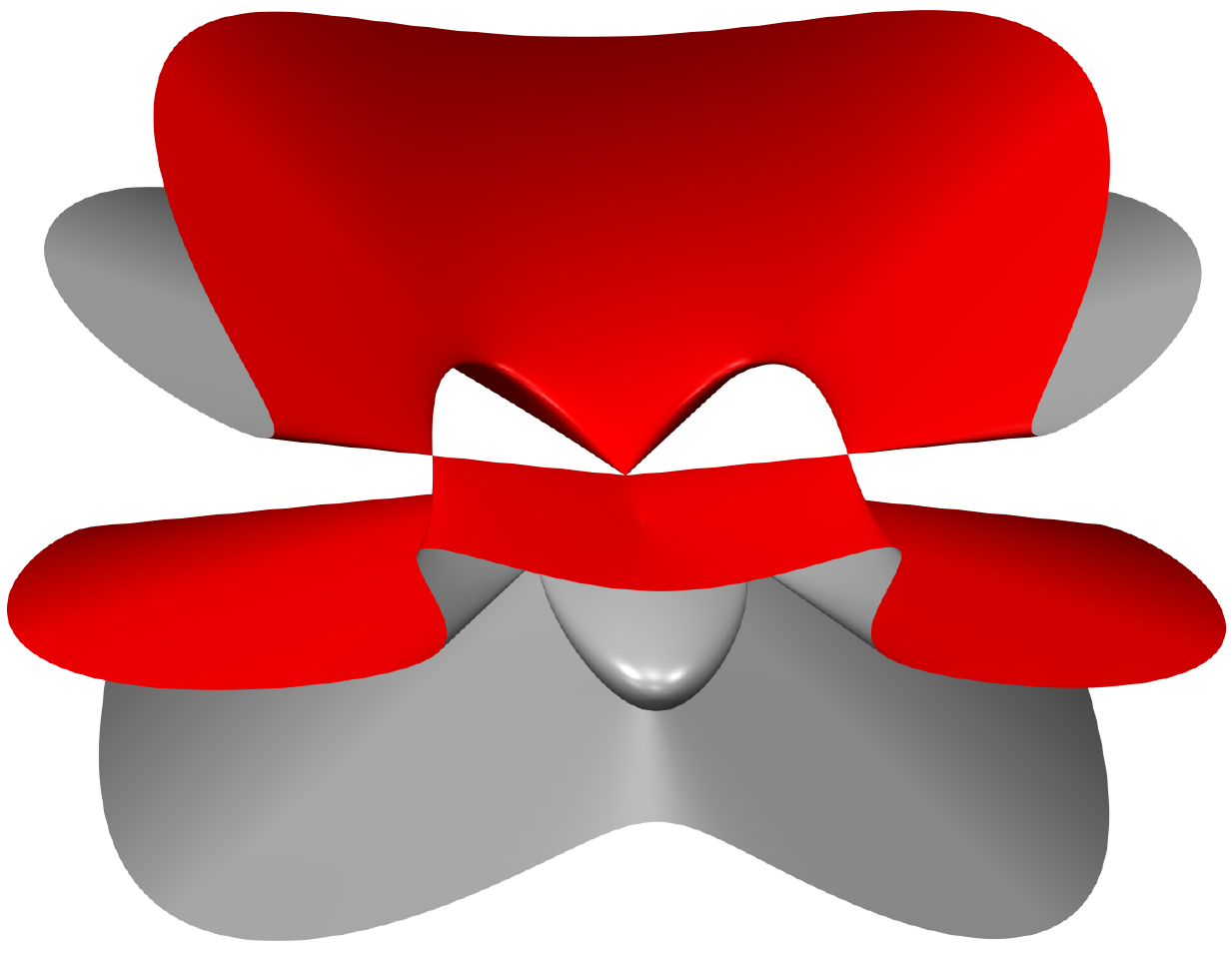}
    \hspace{2em}
    \includegraphics[height = 0.18\textheight]{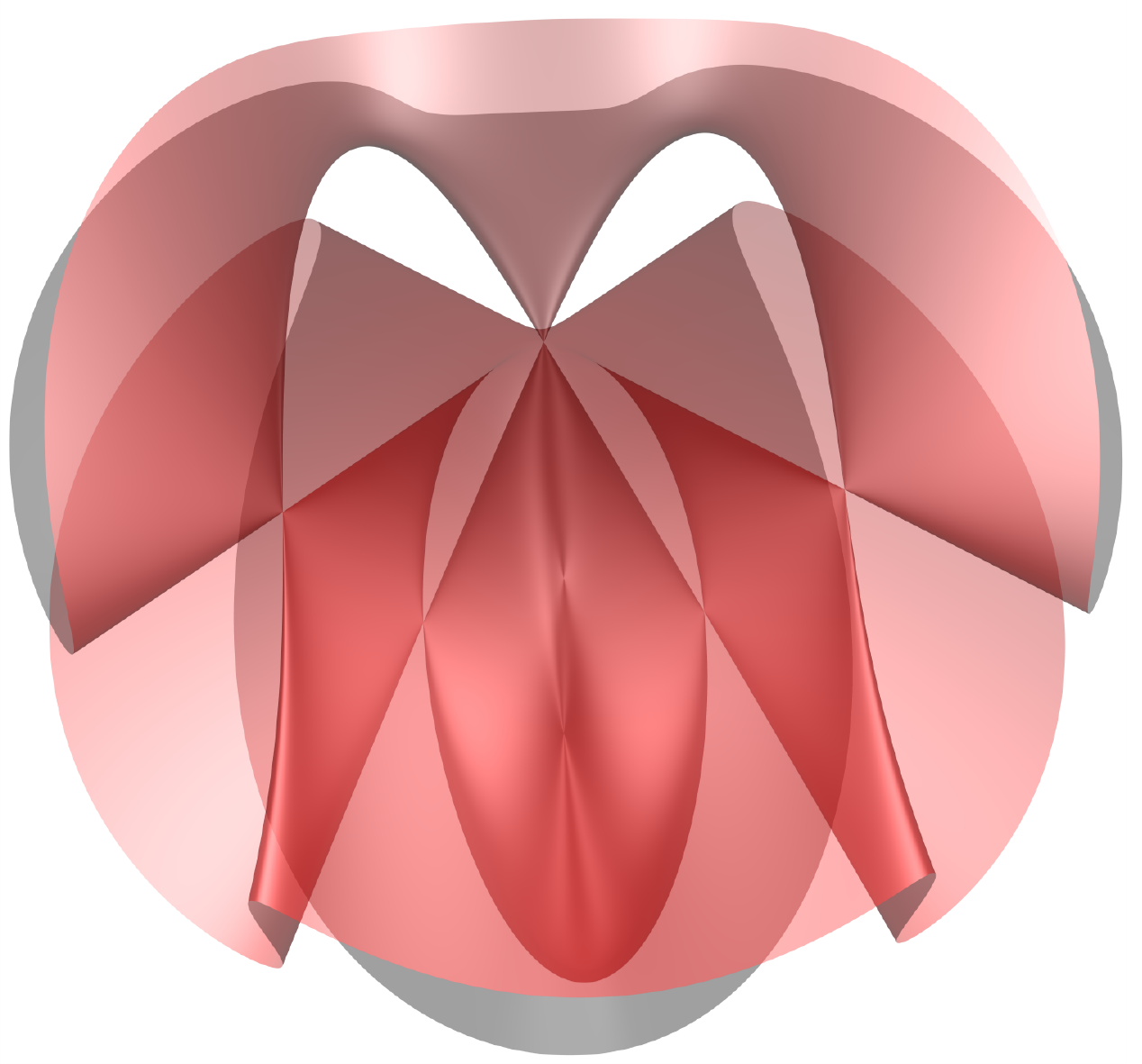}
    \addtolength{\belowcaptionskip}{-1em}
    \caption{A symmetroid of type~\ref{main-item:triple} with $(a, b) = (6, 4)$.}
    \label{fig:triple-point}
\end{figure}

\begin{figure}[!hbtp]
    \centering
    \includegraphics[height = 0.18\textheight]{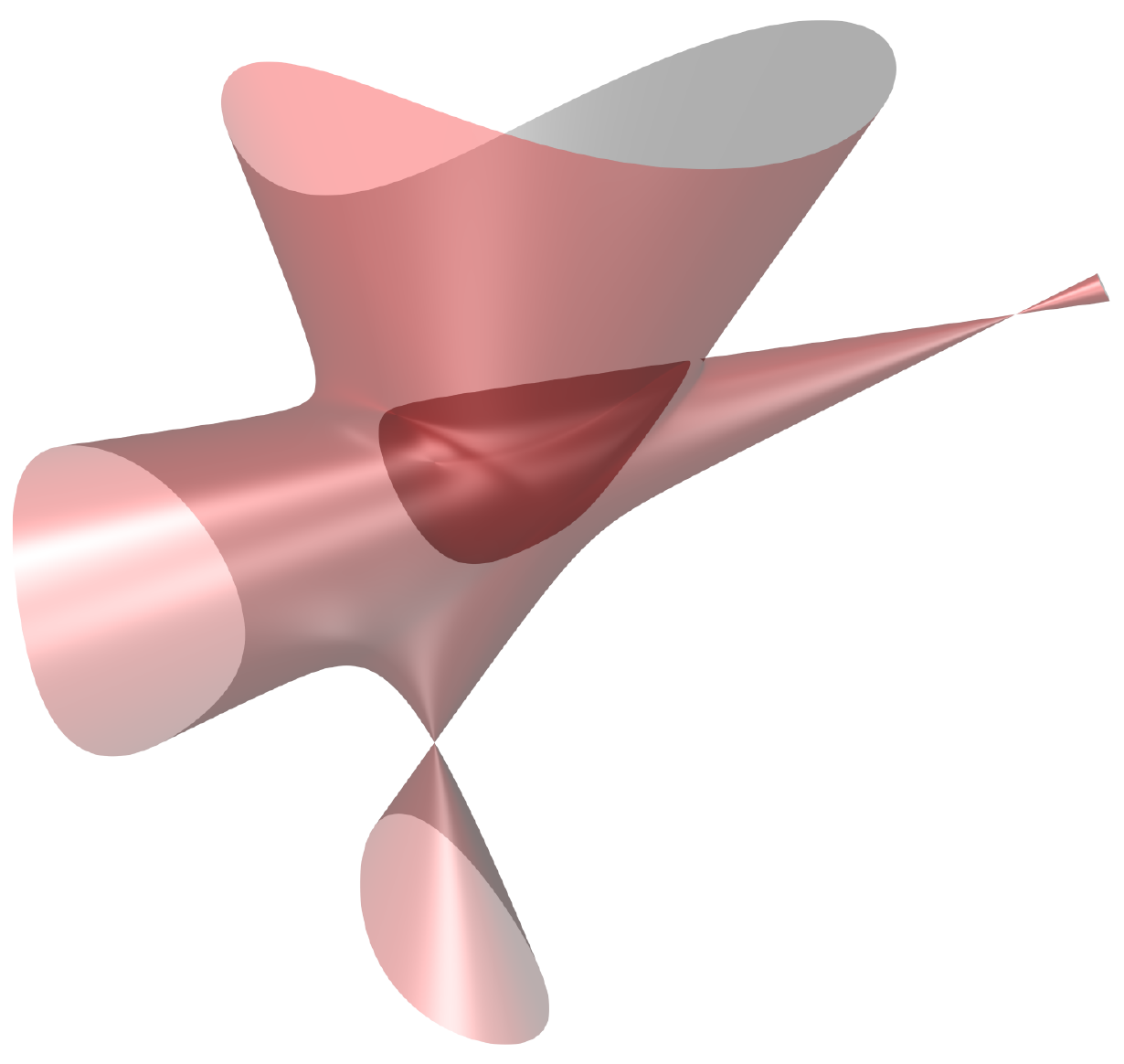}
    \hspace{2em}
    \includegraphics[height = 0.18\textheight]{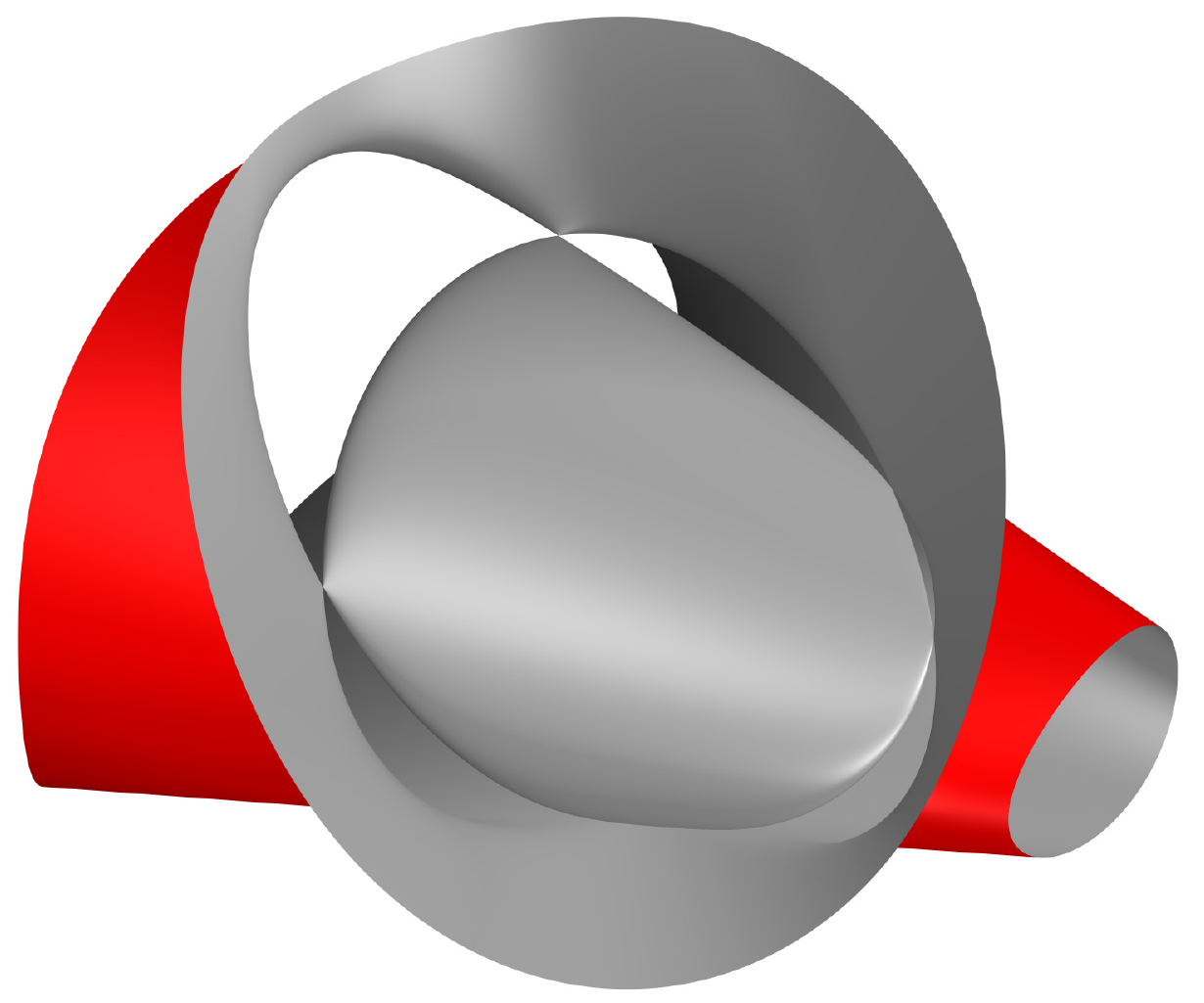}
    \addtolength{\belowcaptionskip}{-1em}
    \caption{A symmetroid of type~\ref{main-item:tacnode-on-boundary} with $(a, b) = (6, 4)$.}
    \label{fig:tacnode}
\end{figure}

\begin{figure}[h!tbp]
    \centering
    \includegraphics[height = 0.18\textheight]{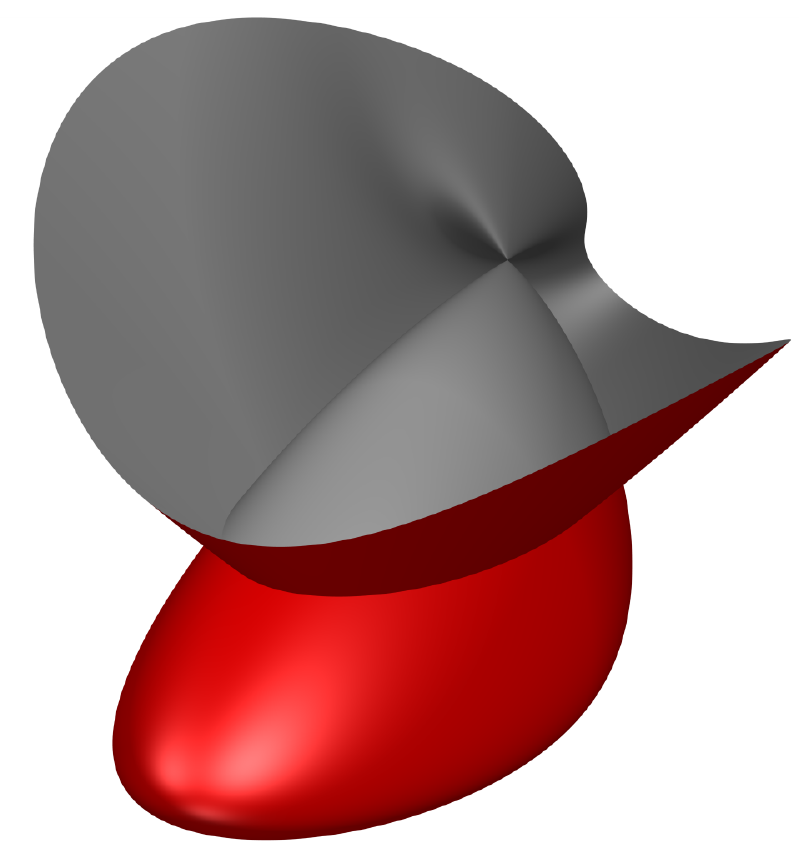}
    \hspace{2em}
    \includegraphics[height = 0.18\textheight]{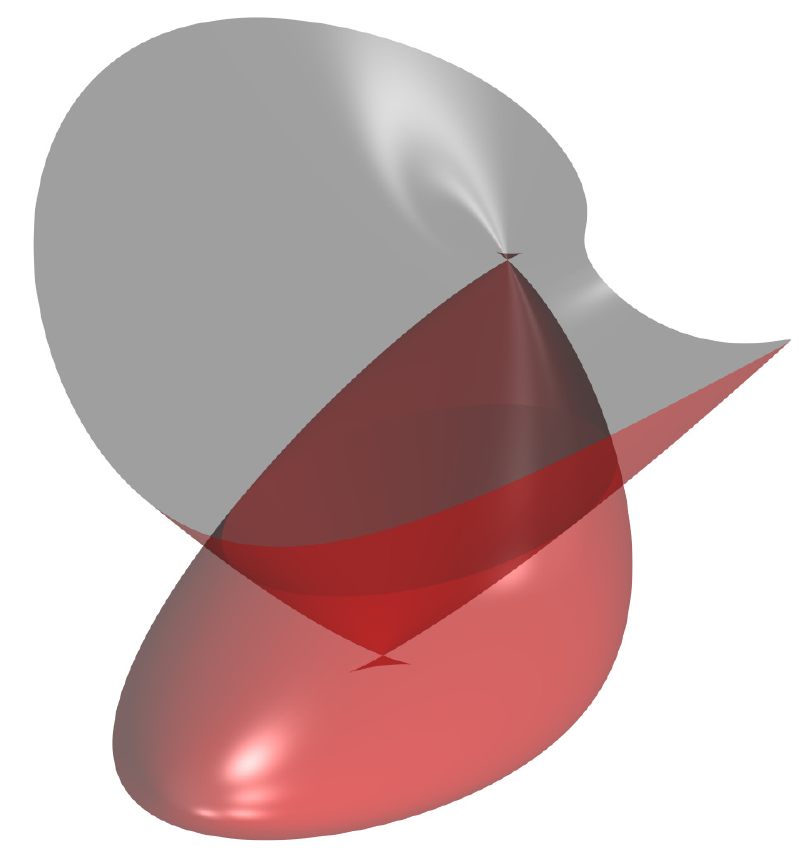}
    \addtolength{\belowcaptionskip}{-1em}
    \caption{A symmetroid of type~\ref{main-item:conic-on-boundary} with $(a, b) = (2, 2)$.}
    \label{fig:boundary-conic}
\end{figure}

\begin{figure}[h!tbp]
    \centering
    \includegraphics[height = 0.18\textheight]{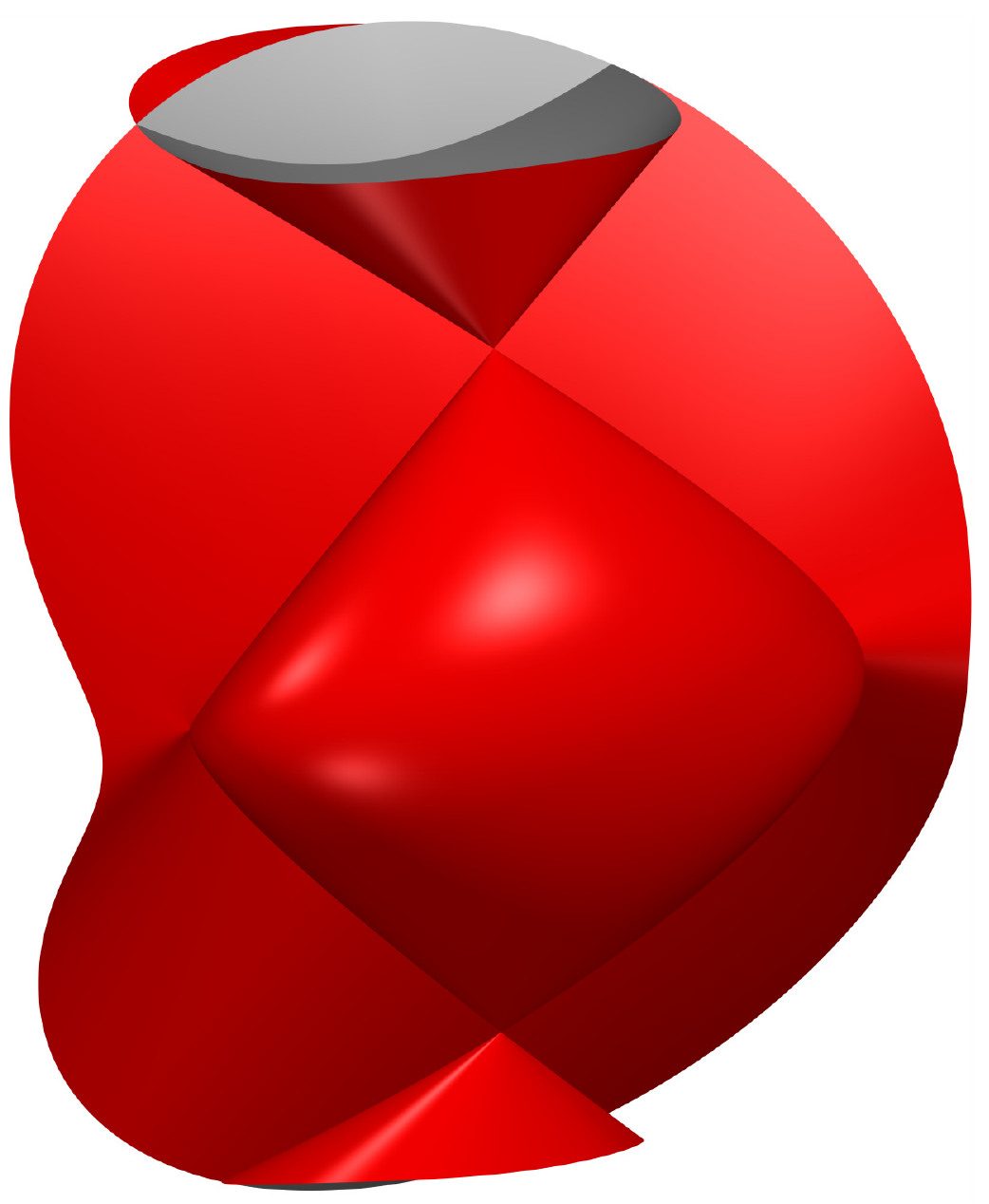}
    \addtolength{\belowcaptionskip}{-2em}
    \caption{A symmetroid of type~\ref{main-item:conic-disjoint} with $(a, b) = (4, 4)$.}
    \label{fig:disjoint-conic}
\end{figure}

\begin{figure}[h!t]
    \centering
    \includegraphics[width = 0.4\textwidth]{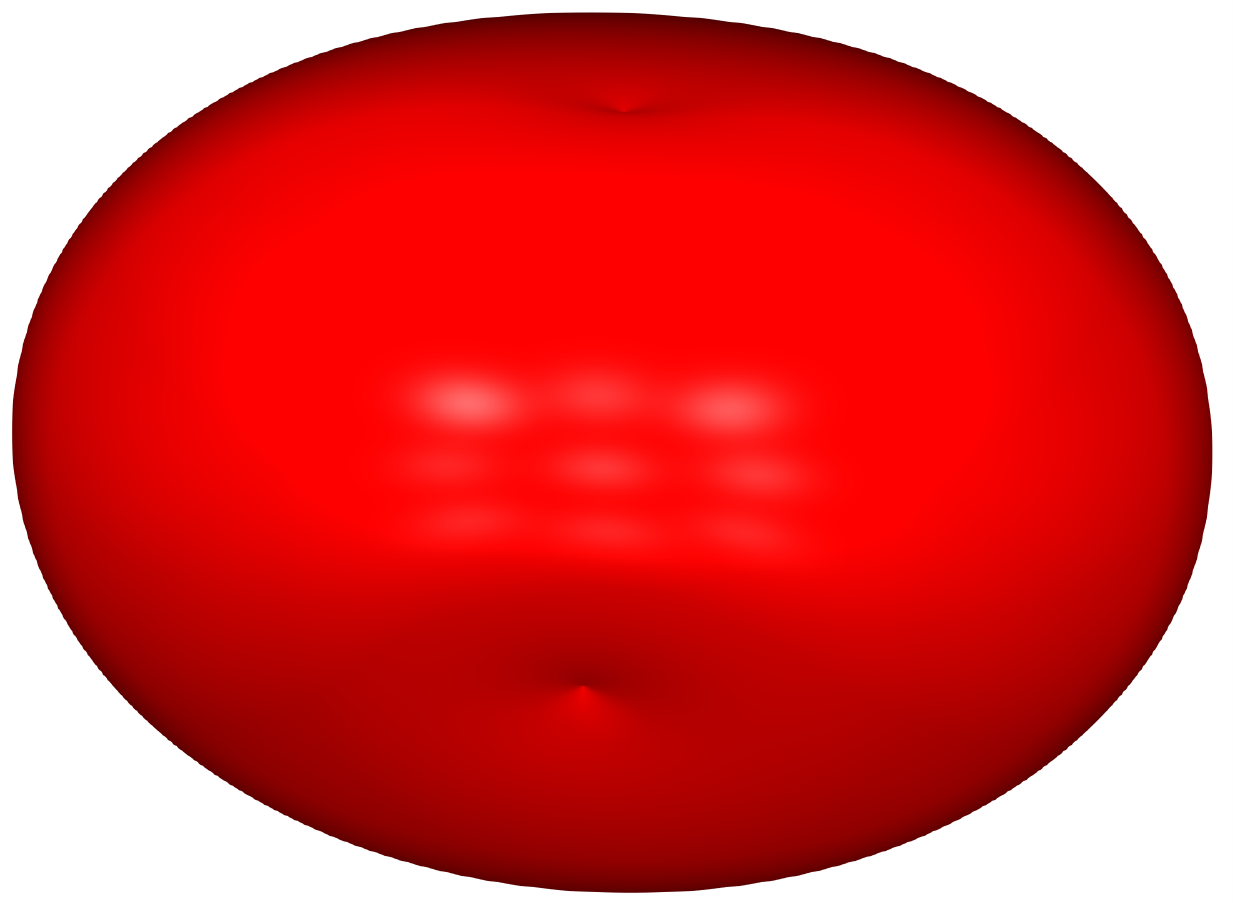}
    \hspace{2em}
    \includegraphics[width = 0.4\textwidth]{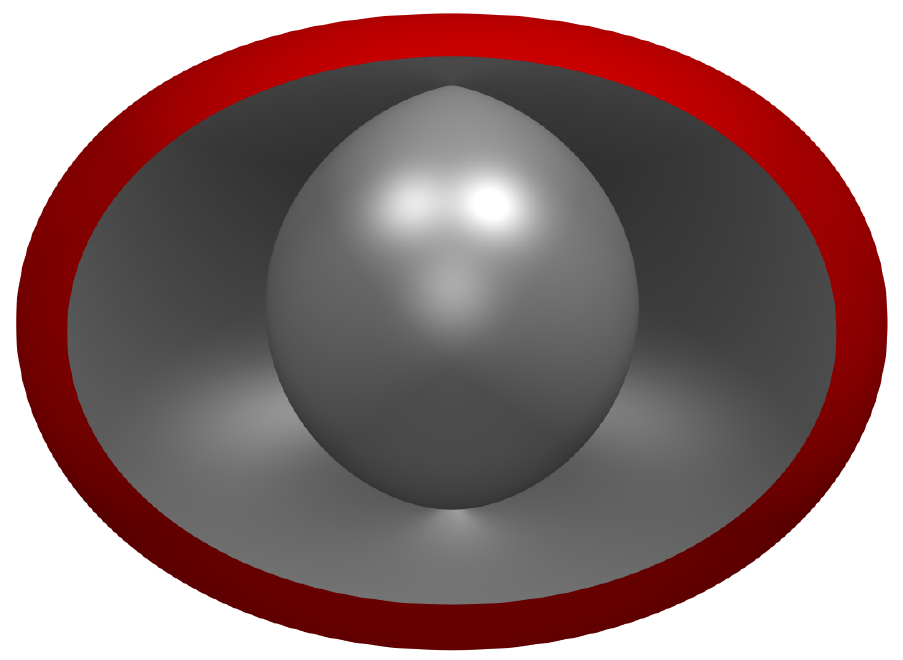}
    \caption{A symmetroid of type~\ref{main-item:cyclide} with $(a, b) = (2, 2)$.
             The surface is known as a \enquote{spindle cyclide}.}
    \label{fig:cyclide}
\end{figure}

\begin{figure}[!bhtp]
    \centering
    \includegraphics[height = 0.16\textheight]{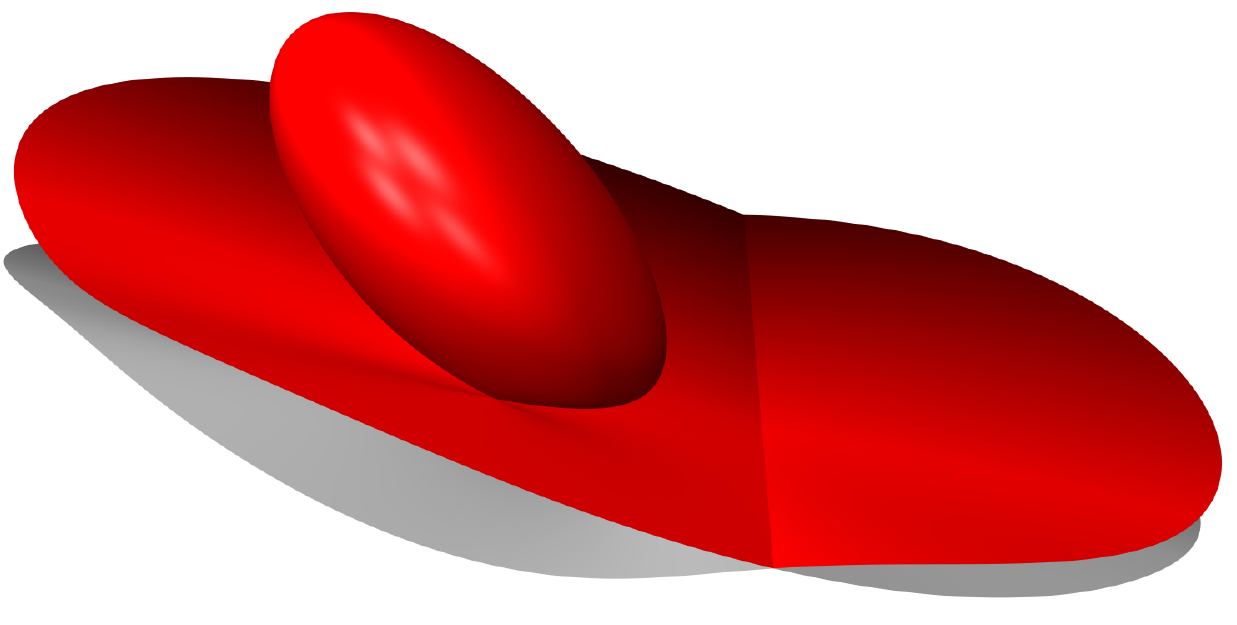}
    \hspace{2em}
    \includegraphics[height = 0.16\textheight]{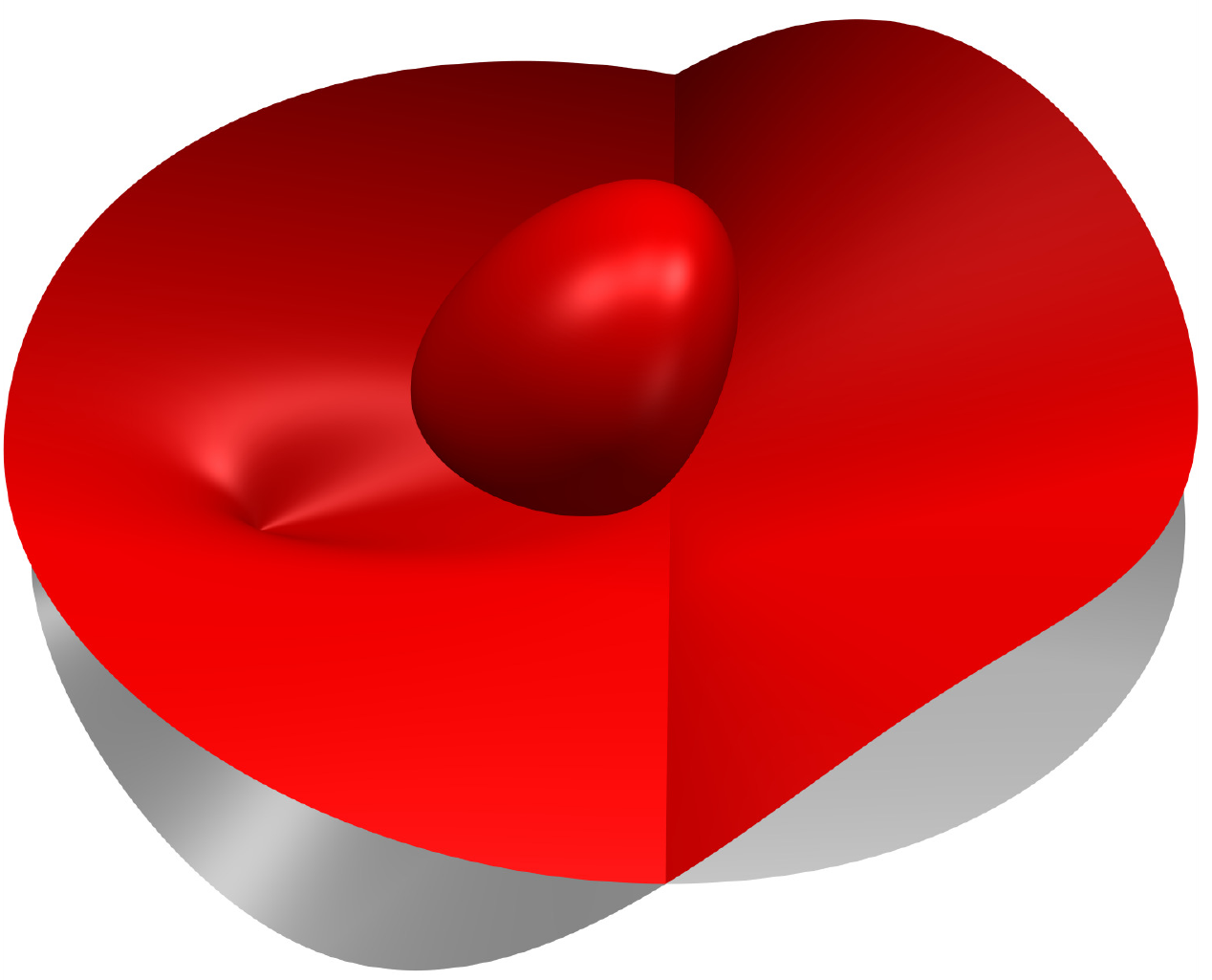}
    \caption{Two symmetroids of type~\ref{main-item:line}.
             The surface to the left has $(a, b) = (2, 2)$
             and the surface to the right has $(a, b) = (2, 0)$.}
    \label{fig:one-line}
\end{figure}

\begin{figure}[!thbp]
    \centering
    \includegraphics[height = 0.16\textheight]{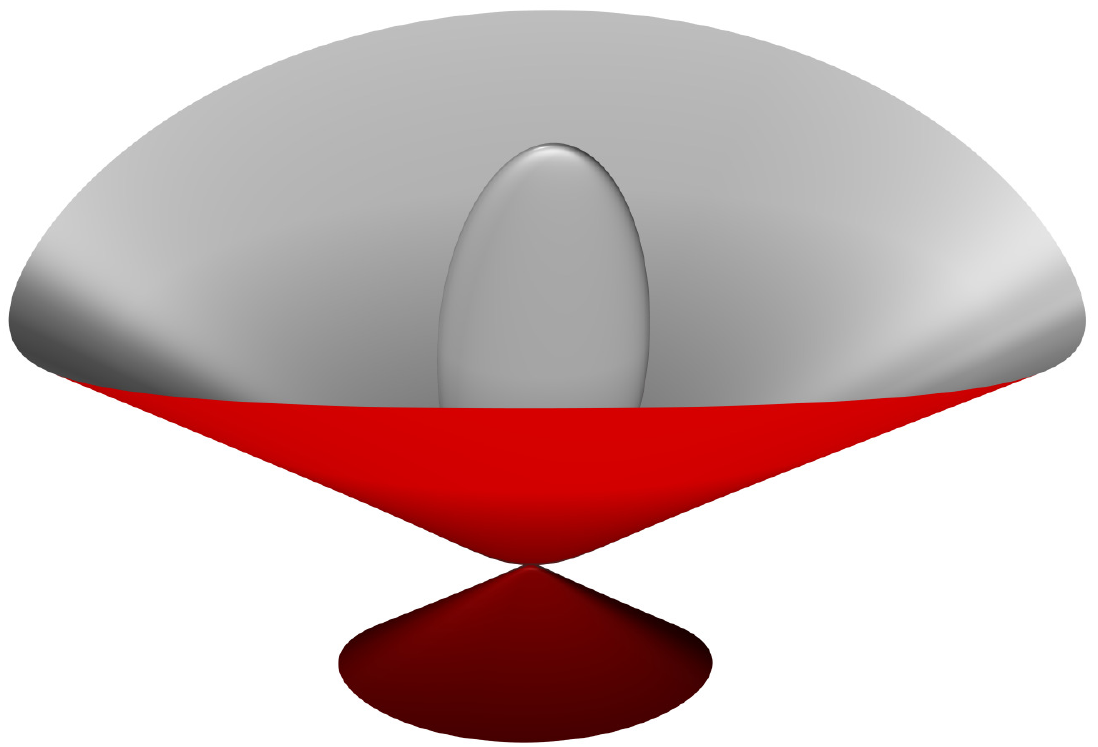}
    \hspace{2em}
    \includegraphics[height = 0.16\textheight]{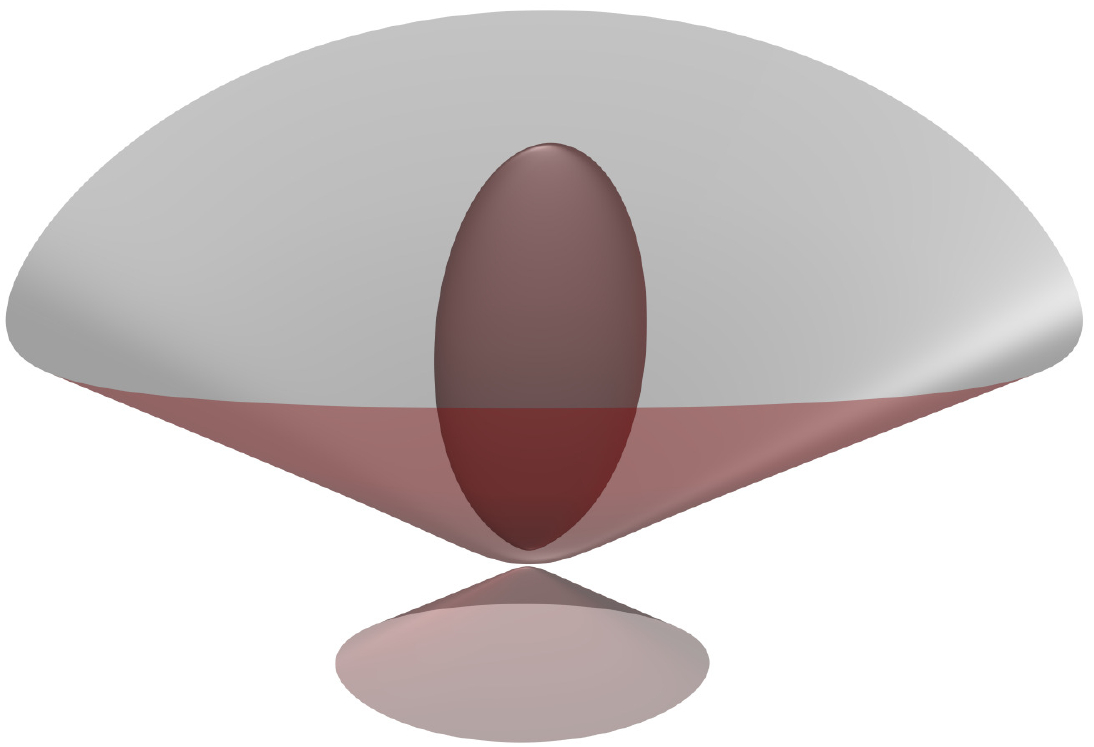}
    \caption{A symmetroid of type~\ref{main-item:two-lines-on-boundary}
             with $(a, b) = (0, 0)$.}
\end{figure}

\begin{figure}[!htbp]
    \centering
    \includegraphics[height = 0.2\textheight]{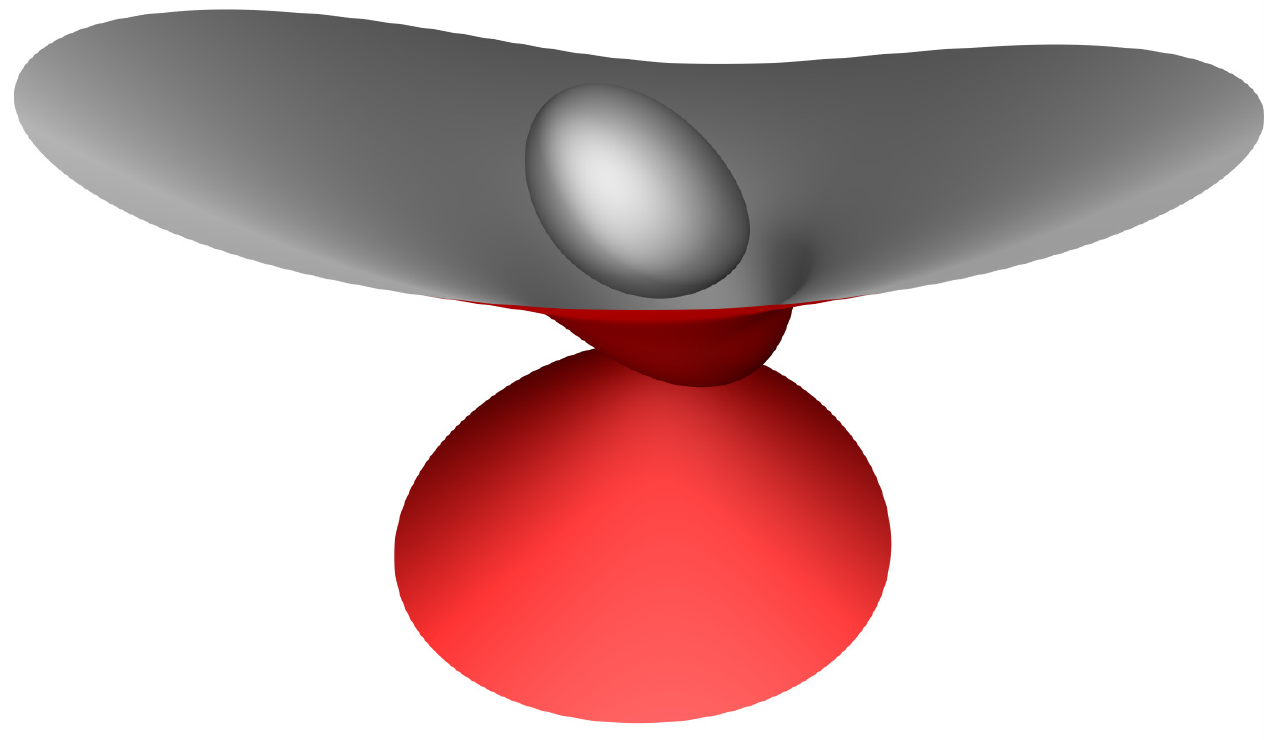}
    \caption{A symmetroid of type~\ref{main-item:two-lines-disjoint}
             with $(a, b) = (0, 0)$.}
\end{figure}

\printbibliography

\paragraph{Authors' addresses}
\begin{description}
    \item[Martin Hels{\o}]
    University of Oslo,
    Postboks 1053 Blindern,
    0316 Oslo,
    Norway,
    \href{mailto:martibhe@math.uio.no}{\nolinkurl{martibhe@math.uio.no}}

    \item[Kristian Ranestad]
    University of Oslo,
    Postboks 1053 Blindern,
    0316 Oslo,
    Norway,
    \href{mailto:ranestad@math.uio.no}{\nolinkurl{ranestad@math.uio.no}}
\end{description}

\end{document}